\title{Induced Ramsey-type theorems}
\author{
Jacob Fox\thanks{Department of Mathematics, Princeton, Princeton, NJ.
Email: {\tt
jacobfox@math.princeton.edu}. Research supported by an NSF Graduate
Research Fellowship and a Princeton Centennial Fellowship.} \and
Benny Sudakov\thanks{Department of Mathematics, Princeton,
Princeton, NJ. Email: {\tt bsudakov@math.princeton.edu}. Research
supported in part by NSF CAREER award DMS-0546523, NSF grant
DMS-0355497 and by a USA-Israeli BSF grant.}}

\documentclass[11pt]{article}
\usepackage{amsfonts,amssymb,amsmath,latexsym}

\newenvironment{proof}
      {\medskip\noindent{\bf Proof.}\hspace{1mm}}
      {\hfill$\Box$\medskip}

\def\qed{\ifvmode\mbox{ }\else\unskip\fi\hskip 1em plus 10fill$\Box$}

\newtheorem{theorem}{Theorem}[section]

\newtheorem{lemma}[theorem]{Lemma}
\newtheorem{proposition}[theorem]{Proposition}
\newtheorem{corollary}[theorem]{Corollary}
\newtheorem{conjecture}[theorem]{Conjecture}

\newtheorem{claim}[theorem]{Claim}

\newtheorem{definition}[theorem]{Definition}

\begin{document}
\date{}

\maketitle

\begin{abstract}
We present a unified approach to proving Ramsey-type theorems for
graphs with a forbidden induced subgraph which can be used to extend
and improve the earlier results of R\"odl, Erd\H{o}s-Hajnal,
Pr\"omel-R\"odl, Nikiforov, Chung-Graham, and \L uczak-R\"odl. The
proofs are based on a simple lemma (generalizing one by Graham,
R\"odl, and Ruci\'nski) that can be used as a replacement for
Szemer\'edi's regularity lemma, thereby giving much better bounds.
The same approach can be also used to show that pseudo-random graphs
have strong induced Ramsey properties. This leads to explicit
constructions for upper bounds on various induced Ramsey numbers.
\end{abstract}

\section{Background and Introduction}

Ramsey theory refers to a large body of deep results in mathematics
concerning partitions of large structures. Its underlying philosophy
is captured succinctly by the statement that ``In a large system,
complete disorder is impossible.'' This is an area in which a great
variety of techniques from many branches of mathematics are used and
whose results are important not only to graph theory and
combinatorics but also to logic, analysis, number theory, and
geometry. Since the publication of the seminal paper of Ramsey
\cite{Ra} in 1930, this subject has grown with increasing vitality,
and is currently among the most active areas in combinatorics.

For a graph $H$, the {\it Ramsey number} $r(H)$ is the least
positive integer $n$ such that every two-coloring of the edges of
the complete graph $K_n$ on $n$ vertices contains a monochromatic copy
of $H$. Ramsey's theorem states that $r(H)$ exists for every graph
$H$. A classical result of Erd\H{o}s and Szekeres~\cite{ErSz}, which
is a quantitative version of Ramsey's theorem, implies that $r(K_k)
\leq 2^{2k}$ for every positive integer $k$. Erd\H{o}s~\cite{Er}
showed using probabilistic arguments that $r(K_k) > 2^{k/2}$ for $k
> 2$. Over the last sixty years, there has been several
improvements on the lower and upper bounds of $r(K_k)$, the most
recent by Conlon \cite{Co}. However, despite efforts by various
researchers, the constant factors in the exponents of these bounds
remain the same.

A subset of vertices of a graph is {\it homogeneous} if it is either
an independent set (empty subgraph) or a clique (complete subgraph).
For a graph $G$, denote by $\hom(G)$ the size of the largest
homogeneous subset of vertices of $G$. A restatement of the
Erd\H{o}s-Szekeres result is that every graph $G$ on $n$ vertices
satisfies $\hom(G) \geq \frac{1}{2}\log n$, while the Erd\H{o}s
result says that for each $n\geq 2$ there is a graph $G$ on $n$
vertices with $\hom(G) \leq 2\log n$. (Here, and throughout the
paper, all logarithms are base $2$.) The only known proofs of the
existence of {\it Ramsey graphs}, i.e., graphs for which
$\hom(G)=O(\log n)$, come from various models of random graphs with
edge density bounded away from $0$ and $1$. This supports the belief
that any graph with small $\hom(G)$ looks `random' in one sense or
another. There are now several results which show that Ramsey graphs
have random-like properties.

A graph $H$ is an {\it induced subgraph} of a graph $G$ if $V(H)
\subset V(G)$ and two vertices of $H$ are adjacent if and only if
they are adjacent in $G$. A graph is {\it $k$-universal} if it
contains all graphs on at most $k$ vertices as induced subgraphs. A
basic property of large random graphs is that they almost surely are
$k$-universal. There is a general belief that graphs which are not
$k$-universal are highly structured. In particular, they should
contain a homogeneous subset which is much larger than that
guaranteed by the Erd\H{o}s-Szekeres bound for general graphs.

In the early 1970's, an important generalization of Ramsey's
theorem, known as the Induced Ramsey Theorem, was discovered
independently by Deuber \cite{De}, Erd\H{o}s, Hajnal, and Posa
\cite{ErHaPo}, and R\"odl \cite{Ro1}. It states that for every graph
$H$ there is a graph $G$ such that in every $2$-edge-coloring of $G$ there is
an induced copy of $H$ whose edges are monochromatic. The least positive integer
$n$ for which there is an $n$-vertex graph with this property is called the {\it induced
Ramsey number} $r_{\textrm{ind}}(H)$. All of the early proofs of the Induced
Ramsey Theorem give enormous upper bounds on $r_{\textrm{ind}}(H)$.
It is still a major open problem to prove good bounds on induced
Ramsey numbers. Ideally, we would like to understand conditions for
a graph $G$ to have the property that in every two-coloring of the
edges of $G$, there is an induced copy of graph $H$ that is
monochromatic.

In this paper, we present a unified approach to proving Ramsey-type
theorems for graphs with a forbidden induced subgraph which can be
used to extend and improve results of various researchers. The same
approach is also used to prove new bounds on induced Ramsey numbers.
In the few subsequent sections we present in full detail our
theorems and compare them with previously obtained results.

\subsection{Ramsey properties of $H$-free graphs}\label{hfree}

As we already mentioned, there are several results (see, e.g.,
\cite{ErSzem,She,AlKrSu,BuSu}) which indicate that Ramsey graphs,
graphs $G$ with relatively small $\hom(G)$, have random-like
properties. The first advance in this area was made by Erd\H{o}s and
Szemer\'edi \cite{ErSzem}, who showed that the Erd\H{o}s-Szekeres
bound $\hom(G) \geq \frac{1}{2}\log n$ can be improved for graphs
which are very sparse or very dense. The edge density of a graph $G$
is the fraction of pairs of distinct vertices of $G$ that are edges.
The Erd\H{o}s-Szemer\'edi theorem states that there is an absolute
positive constant $c$ such that $\hom(G) \geq \frac{c \log
n}{\epsilon \log \frac{1}{\epsilon}}$ for every graph $G$ on $n$
vertices with edge density $\epsilon \in (0,1/2)$. This result shows
that the Erd\H{o}s-Szekeres bound can be significantly improved for
graphs that contain a large subset of vertices that is very sparse
or very dense.

R\"odl \cite{Ro} proved
that if a graph is not $k$-universal with $k$
fixed, then it contains a linear-sized induced subgraph that is very
sparse or very dense.  A graph is called {\it $H$-free} if it does not
contain $H$ as an induced subgraph. More precisely, R\"odl's theorem says that for
each graph $H$ and $\epsilon \in (0,1/2)$, there is a positive
constant $\delta(\epsilon,H)$ such that every $H$-free graph on $n$
vertices contains an induced subgraph on at least
$\delta(\epsilon,H)n$ vertices with edge density either at most
$\epsilon$ or at least $1-\epsilon$. Together with the theorem of Erd\H{o}s and Szemeredi, it
shows that the Erd\H{o}s-Szekeres bound can be improved by any constant factor
for any family of graphs that have a forbidden induced subgraph.

R\"odl's proof uses Szemer\'edi's regularity lemma \cite{Sz}, a
powerful tool in graph theory, which was introduced by Szemer\'edi
 in his celebrated proof of the Erd\H{o}s-Tur\'an conjecture on long arithmetic progressions in
dense subsets of the integers. The regularity lemma roughly says
that every large graph can be partitioned into a small number of
parts such that the bipartite subgraph between almost every pair of
parts is random-like. To properly state the regularity lemma
requires some terminology. The edge density $d(X,Y)$ between two
subsets of vertices of a graph $G$ is the fraction of pairs $(x,y)
\in X \times Y$ that are edges of $G$, i.e.,
$d(X,Y)=\frac{e(X,Y)}{|X||Y|}$, where $e(X,Y)$ is the number of edges with one endpoint in $X$ and
the other in $Y$.
A pair $(X,Y)$ of vertex sets
is called $\epsilon$-regular if for every $X' \subset X$ and $Y'
\subset Y$ with $|X'| > \epsilon |X|$ and $|Y'| > \epsilon |Y|$ we
have $|d(X',Y')-d(X,Y)|<\epsilon$. A partition $V=\bigcup_{i=1}^k
V_i$ is called {\it equitable} if $\big||V_i|-|V_j|\big|\leq 1$ for
all $i,j$.

Szemer\'edi's regularity lemma \cite{Sz} states that for each
$\epsilon>0$, there is a positive integer $M(\epsilon)$ such that
the vertices of any graph $G$ can be equitably partitioned
$V(G)=\bigcup_{i=1}^k V_i$ into $k$ subsets with $\epsilon^{-1}
\leq k \leq M(\epsilon)$ satisfying that all but at most $\epsilon
k^2$ of the pairs $(V_i,V_j)$ are $\epsilon$-regular. For more
background on the regularity lemma, see the excellent
survey by Koml\'os and Simonovits \cite{KoSi}.

In the regularity lemma, $M(\epsilon)$ can be taken to be a tower of
$2$'s of height proportional to $\epsilon^{-5}$. On the other hand,
Gowers \cite{Go} proved a lower bound on $M(\epsilon)$ which is a
tower of $2$'s of height proportional to $\epsilon^{-\frac{1}{16}}$.
His result demonstrates that $M(\epsilon)$ is inherently large as a
function of $\epsilon^{-1}$. Unfortunately, this implies that the
bounds obtained by applications of the regularity lemma are often
quite poor. In particular, this is a weakness of the bound on
$\delta(\epsilon,H)$ given by R\"odl's proof of his theorem. It is
therefore desirable to find a new proof of R\"odl's theorem that
does not use the regularity lemma. The following theorem does just
that, giving a much better bound on $\delta(\epsilon,H)$. Its proof
works as well in a multicolor setting (see concluding remarks).

\begin{theorem} \label{main}
There is a constant $c$ such that for each $\epsilon \in (0,1/2)$
and graph $H$ on $k \geq 2$ vertices, every $H$-free graph on $n$
vertices contains an induced subgraph on at least $2^{-ck(\log
\frac{1}{\epsilon})^2}n$ vertices with edge density either at most
$\epsilon$ or at least $1-\epsilon$.
\end{theorem}

Nikiforov \cite{Ni} recently strengthened R\"odl's theorem by
proving that for each $\epsilon>0$ and graph $H$ of order $k$, there
are positive constants $\kappa=\kappa(\epsilon,H)$ and
$C=C(\epsilon,H)$ such that for every graph $G=(V,E)$ that
contains at most $\kappa|V|^{k}$ induced copies of $H$, there is an
equitable partition $V=\bigcup_{i=i}^C V_i$ of the vertex set such
that the edge density in each $V_i$ ($i \geq 1$) is at most
$\epsilon$ or at least $1-\epsilon$. Using the same technique as the
proof of Theorem \ref{main}, we give a new proof of this result
without using the regularity lemma, thereby solving the main open problem
posed in \cite{Ni}.

Erd\H{o}s and Hajnal \cite{ErHa} gave a significant improvement on
the Erd\H{o}s-Szekeres bound on the size of the largest homogeneous
set in $H$-free graphs. They proved that for every graph $H$ there
is a positive constant $c(H)$ such that $\hom(G) \geq
2^{c(H)\sqrt{\log n}}$ for all $H$-free graphs $G$ on $n$ vertices.
Erd\H{o}s and Hajnal further conjectured that every such $G$
contains a complete or empty subgraph of order $n^{c(H)}$. This
beautiful problem has received increasing attention by various
researchers, and was also featured by Gowers \cite{Go1} in his list
of problems at the turn of the century. For various partial results
on the Erd\H{o}s-Hajnal conjecture see, e.g., \cite{AlPaSo, ErHaPa,
FoSu, AlPaPiRaSh, FoPaTo2, LaMaPaTo, ChSa} and their references.

Recall that a graph is $k$-universal if it contains all graphs on at
most $k$ vertices as induced subgraphs. Note that the Erd\H{o}s-Hajnal
bound, in particular, implies that, for every fixed $k$,
sufficiently large Ramsey graphs are $k$-universal. This was
extended further by Pr\"omel and R\"odl \cite{PrRo}, who obtained an
asymptotically best possible result. They proved that if $\hom(G)
\leq c_1 \log n$ then $G$ is $c_2\log n$-universal for some constant
$c_2$ which depends on $c_1$.

Let $\hom(n,k)$ be the largest positive integer such that every
graph $G$ on $n$ vertices is $k$-universal or satisfies $\hom(G)
\geq \hom(n,k)$. The Erd\H{o}s-Hajnal theorem and the Promel-R\"odl
theorem both say that $\hom(n,k)$ is large for fixed or slowly
growing $k$. Indeed, from the first theorem it follows that for
fixed $k$ there is $c(k)>0$ such that $\hom(n,k) \geq
2^{c(k)\sqrt{\log n}}$, while the second theorem says that for each
$c_1$ there is $c_2>0$ such that $\hom(n,c_2\log n) \geq c_1 \log
n$. One would naturally like to have a general lower bound on
$\hom(n,k)$ that implies both the Erd\H{o}s-Hajnal and Promel-R\"odl
results. This is done in the following theorem.

\begin{theorem}\label{combined}
There are positive constants $c_3$ and $c_4$ such that for all
$n,k$, every graph on $n$ vertices is $k$-universal or satisfies
$\hom(G) \geq c_32^{c_4\sqrt{\frac{\log n}{k}}}\log n.$
\end{theorem}

Theorem \ref{main} can be also used to answer a question
of Chung and Graham \cite{ChGr}, which was motivated by the study of quasirandom graphs.
Given a fixed graph $H$, it is well known that a typical graph
on $n$ vertices  contains many induced copies of $H$ as $n$ becomes large.
Therefore if a large graph $G$ contains no induced copy of $H$, its edge distribution should deviate from
``typical" in a rather strong way. This intuition was made rigorous in \cite{ChGr}, where the authors
proved that if a graph $G$ on $n$ vertices is not $k$-universal,
then there is a subset $S$ of $\lfloor \frac{n}{2} \rfloor$ vertices
of $G$ such that $|e(S)-\frac{1}{16}n^2|>2^{-2(k^2+27)}n^2$. For
positive integers $k$ and $n$, let $D(k,n)$ denote the largest
integer such that every graph $G$ on $n$ vertices that is not
$k$-universal contains a subset $S$ of vertices of size $\lfloor
\frac{n}{2} \rfloor$ with $|e(S)-\frac{1}{16}n^2|>D(k,n)$. Chung and
Graham asked whether their lower bound on $D(k,n)$ can be
substantially improved, e.g., replaced by $c^{-k}n^2$.
Using Theorem \ref{main} this can be easily done as follows.

A lemma of Erd\H{o}s, Goldberg, Pach, and Spencer \cite{ErGoPaSp}
implies that if a graph on $n$ vertices has a subset $R$ that
deviates by $D$ edges from having edge density $1/2$, then there is
a subset $S$ of size $\lfloor n/2 \rfloor$ that deviates by at least a
constant times $D$ edges from having edge density $1/2$. By Theorem
\ref{main} with $\epsilon=1/4$, there is a positive constant $C$
such that every graph on $n$ vertices that is not $k$-universal has
a subset $R$ of size at least $C^{-k}n$ with edge density at most
$1/4$ or at least $3/4$. This $R$
deviates from having edge density $1/2$ by at least
$$\frac{1}{4}{|R| \choose 2}\geq \frac{1}{16}|R|^2 \geq
\frac{1}{16}C^{-2k}n^2$$ edges. Thus, the above mentioned lemma from
\cite{ErGoPaSp} implies that there is an absolute constant $c$ such
that every graph $G$ on $n$ vertices which is not $k$-universal
contains a subset $S$ of size $\lfloor n/2 \rfloor$ with
$|e(S)-\frac{n^2}{16}|>c^{-k}n^2$. Chung and Graham also ask for
non-trivial upper bounds on $D(k,n)$. In this direction, we show
that there are $K_k$-free graphs on $n$ vertices for which
$|e(S)-\frac{1}{16}n^2|=O(2^{-k/4}n^2)$ holds for every subset $S$
of $\lfloor \frac{n}{2} \rfloor$ vertices of $G$. Together with the
lower bound it determines the asymptotic behavior of $D(k,n)$ and
shows that there are constants $c_1,c_2>1$ such that
$c_1^{-k}n^2<D(k,n)<c_2^{-k}n^2$ holds for all positive integers $k$
and $n$. This completely answers the questions of Chung and Graham.

Moreover, we can obtain a more precise result about the relation
between the number of induced copies of a fixed graph $H$ in a large graph $G$
and the edge distribution of $G$.
In their celebrated paper, Chung, Graham, and Wilson
\cite{ChGrWi} introduced a large collection of equivalent graph
properties shared by almost all graphs which are called {\it
quasirandom}. For a graph $G=(V,E)$ on $n$ vertices, two of these
properties are

\vspace{0.1cm}
${\bf P_1}:\textrm{For each subset}~S \subset V,$
$$e(S)=\frac{1}{4}|S|^2+o(n^2).$$

${\bf P}_2$: For every fixed graph $H$ with $k$ vertices, the number
of labeled induced copies of $H$ in $G$ is
 $$(1+o(1))n^{k}2^{-{k \choose 2}}.$$
So one can ask naturally, by how much does a graph deviate from ${\bf
P}_1$ assuming a deviation from ${\bf P}_2$? The following theorem
answers this question.

\begin{theorem}\label{dev} Let $H$ be a graph with $k$ vertices and $G=(V,E)$ be a
graph with $n$ vertices and at most $(1-\epsilon)2^{-{k \choose
2}}n^k$ labeled induced copies of $H$. Then there is a subset $S
\subset V$ with $|S|=\lfloor n/2 \rfloor$ and $|e(S)-\frac{n^2}{16}|
\geq \epsilon c^{-k} n^2$, where $c$ is an absolute constant.
\end{theorem}

The proof of Theorem \ref{dev} can be easily adjusted if we
replace the ``at most'' with ``at least'' and the $(1-\epsilon)$
factor by $(1+\epsilon)$.  Note that this theorem answers the
original question of Chung and Graham in a very strong sense.

\subsection{Induced Ramsey numbers}\label{inducedsubsection}

Recall that the induced Ramsey number $r_{\textrm{ind}}(H)$ is the
minimum $n$ for which there is a graph $G$ with $n$ vertices such
that for every $2$-edge-coloring of $G$, one can find an induced
copy of $H$ in $G$ whose edges are monochromatic. One of the
fundamental results in graph Ramsey theory (see chapter 9.3 of
\cite{Di}), the Induced Ramsey Theorem, says that
$r_{\textrm{ind}}(H)$ exists for every graph $H$. R\"odl \cite{Ro}
noted that a relatively simple proof of the theorem follows from a
simple application of his result discussed in the previous section.
However, all of the early proofs of the Induced Ramsey Theorem give
poor upper bounds on $r_{\textrm{ind}}(H)$.

Since these early proofs, there has been a considerable amount of
research on induced Ramsey numbers. Erd\H{o}s \cite{Er2} conjectured
that there is a constant $c$ such that every graph $H$ on $k$
vertices satisfies $r_{\textrm{ind}}(H) \leq 2^{ck}$. Erd\H{o}s and
Hajnal \cite{Er1} proved that $r_{\textrm{ind}}(H) \leq
2^{2^{k^{1+o(1)}}}$ holds for every graph $H$ on $k$ vertices.
Kohayakawa, Pr\"omel, and R\"odl \cite{KoPrRo} improved this bound
substantially and showed that if a graph $H$ has $k$ vertices and
chromatic number $\chi$, then $r_{\textrm{ind}}(H) \leq k^{ck(\log
\chi)},$ where $c$ is a universal constant. In particular, their result
implies an upper bound of $2^{ck (\log k)^2}$ on the induced Ramsey
number of any graph on $k$ vertices. In their proof, the graph $G$
which gives this bound is randomly constructed using projective
planes.

There are several known results that provide upper bounds on induced
Ramsey numbers for sparse graphs. For example, Beck \cite{Be}
studied the case when $H$ is a tree; Haxell, Kohayakawa, and \L
uczak \cite{HaKoLu} proved that the cycle of length $k$ has induced
Ramsey number linear in $k$; and, settling a conjecture of Trotter,
\L uczak and R\"odl \cite{LuRo} showed that the induced Ramsey
number of a graph with bounded degree is at most polynomial in the
number of its vertices. More precisely, they proved  that for every
integer $d$, there is a constant $c_d$ such that every graph $H$ on
$k$ vertices and maximum degree at most $d$ satisfies
$r_{\textrm{ind}}(H) \leq k^{c_d}$. Their proof, which also uses
random graphs, gives an upper bound on $c_d$ that is a tower of
$2$'s of height proportional to $d^2$.

As noted by Schaefer and Shah \cite{SchSh}, all known proofs of the
Induced Ramsey Theorem either rely on taking $G$ to be an
appropriately chosen random graph or give a poor upper bound on
$r_{\textrm{ind}}(H)$. However, often in combinatorics, explicit
constructions are desirable in addition to existence proofs given by
the probabilistic method. For example, one of the most famous such
problems was posed by Erd\H{o}s \cite{AlSp}, who asked for the
explicit construction of a graph on $n$ vertices without a complete
or empty subgraph of order $ c\log n$. Over the years, this
intriguing problem and its bipartite variant has drawn a lot of attention
by various researches (see, e.g.,
\cite{FrWi,Al2,BaKiShSuWi,Bo,BaRaShWi}), but, despite these efforts,
it is still open. Similarly, one would like to have an explicit
construction for the Induced Ramsey Theorem. We obtain such a
construction using pseudo-random graphs.

The {\it random graph}
$G(n,p)$ is the probability space of all labeled graphs on $n$ vertices, where
every edge appears randomly and independently with probability $p$. An
important property of $G(n,p)$ is that, with high
probability, between any two large subsets of vertices $A$ and $B$, the edge
density $d(A,B)=\frac{e(A,B)}{|A||B|}$ is approximately $p$. This observation is one of the motivations for the following useful
definition. A graph $G=(V,E)$ is {\it $(p,\lambda)$-pseudo-random}
if the following inequality holds for all subsets $A,B \subset V$:
$$|d(A,B)-p| \leq \frac{\lambda}{\sqrt{|A||B|}}.$$
It is easy to show that if $p<0.99$, then with high probability, the
random graph $G(n,p)$ is $(p,\lambda)$-pseudo-random with
$\lambda=O(\sqrt{pn})$. Moreover, there are also many explicit
constructions of pseudo-random graphs which can be obtained using
the following fact. Let $\lambda_1 \geq \lambda_2 \geq \ldots \geq
\lambda_n$ be the eigenvalues of the adjacency matrix of a graph
$G$. An {\it $(n,d,\lambda)$-graph} is a $d$-regular graph on $n$
vertices with $\lambda = \max_{i \geq 2} |\lambda_i|$. It was proved
by Alon (see, e.g., \cite{AlSp}, \cite{KrSu}) that every
$(n,d,\lambda)$-graph is in fact
$(\frac{d}{n},\lambda)$-pseudo-random. Therefore to construct good
pseudo-random graphs we need regular graphs with $\lambda \ll d$.
For more details on pseudo-random graphs, including many
constructions, we refer the interested reader to the recent survey
\cite{KrSu}.

A graph is {\it $d$-degenerate} if every subgraph of it has a vertex
of degree at most $d$. The degeneracy number of a graph $H$ is the
smallest $d$ such that $H$ is $d$-degenerate. This quantity, which
is always bounded by the maximum degree of the graph, is a natural
measure of its sparseness. In particular, in a $d$-degenerate graph
every subset $X$ spans at most $d|X|$ edges. The chromatic number
$\chi(H)$ of graph $H$ is the minimum number of colors needed to
color vertices of $H$ such that adjacent vertices get different
colors. Using a greedy coloring, it is easy to show that
$d$-degenerate graphs have chromatic number at most $d+1$. The
following theorem, which is special case of a more general result
which we prove in Section 4, shows that any sufficiently
pseudo-random graph of appropriate density has strong induced Ramsey
properties.

\begin{theorem}\label{quasicor1}
There is an absolute constant $c$ such that for all integers $k,d,\chi \geq
2$, every $(\frac{1}{k},n^{0.9})$-pseudo-random graph $G$ on $n \geq
k^{cd\log \chi}$ vertices satisfies that every $d$-degenerate graph on
$k$ vertices with chromatic number at most $\chi$ occurs as an
induced monochromatic copy in all $2$-edge-colorings of $G$.
Moreover, all of these induced monochromatic copies can be found in
the same color.
\end{theorem}

This theorem implies that, with high probability, $G(n,p)$ with
$p=1/k$ and $n \geq k^{cd\log \chi}$  satisfies that every
$d$-degenerate graph on $k$ vertices with chromatic number at most
$\chi$ occurs as an induced monochromatic copy in all $2$-edge-colorings
of $G$. It gives the first polynomial upper bound on the induced
Ramsey numbers of $d$-degenerate graphs. In particular, for bounded
degree graphs this is a significant improvement of the above
mentioned \L uczak-R\"odl result. It shows that the exponent of the
polynomial in their theorem can be taken to be $O(d\log d)$, instead
of the previous bound of a tower of $2$'s of height proportional to
$d^2$.

\begin{corollary}\label{corollaryobvious}
There is an absolute constant $c$ such that every $d$-degenerate
graph $H$ on $k$ vertices with chromatic number $\chi \geq 2$ has
induced Ramsey number $r_{\textrm{ind}}(H) \leq k^{cd \log \chi}$.
\end{corollary}

A significant additional benefit of Theorem \ref{quasicor1} is that
it leads to explicit constructions for induced Ramsey numbers. One
such example can be obtained from a construction of Delsarte and
Goethals and also of Turyn (see \cite{KrSu}). Let $r$ be a prime
power and let $G$ be a graph whose vertices are the elements of the
two dimensional vector space over finite field $\mathbb{F}_r$, so
$G$ has $r^2$ vertices. Partition the $r+1$ lines through the origin
of the space into two sets $P$ and $N$, where $|P|=t$. Two vertices
$x$ and $y$ of the graph $G$ are adjacent if $x-y$ is parallel to a
line in $P$. This graph is known to be $t(r-1)$-regular with
eigenvalues, besides the largest one, being either $-t$ or $r-t$.
Taking $t= \frac{r^2}{k(r-1)}$, we obtain an $(n,d,\lambda)$-graph
with $n=r^2$, $d = n/k$, and $\lambda=r-t\leq r \leq n^{1/2}$. This
gives a $(p,\lambda)$-pseudo-random graph with $p =d/n=1/k$ and
$\lambda \leq n^{1/2}$ which satisfies the assertion of Theorem
\ref{quasicor1}.

Another well-known explicit construction is the Paley graph
$P_n$. Let $n$ be a prime power which is congruent to 1 modulo 4 so that $-1$ is a square in
the finite field $\mathbb{F}_n$. The {\it Paley graph} $P_n$ has vertex
set $\mathbb{F}_n$ and distinct elements $x,y \in \mathbb{F}_n$ are
adjacent if $x-y$ is a square. It is well known and not difficult to prove that
the Paley graph $P_n$ is $(1/2,\lambda)$-pseudo-random with $\lambda=\sqrt{n}$.
This can be used together with the generalization of Theorem \ref{quasicor1}, which we discuss in
Section 5, to prove the following result.

\begin{corollary}\label{payley}
There is an absolute constant $c$ such that for prime $n \geq
2^{ck\log^2 k}$, every graph on $k$ vertices occurs as an
induced monochromatic copy in all $2$-edge-colorings of the Paley graph $P_n$.
\end{corollary}

This explicit construction matches the best known upper bound on
induced Ramsey numbers of graphs on $k$ vertices obtained by
Kohayakawa, Pr\"omel, and R\"odl \cite{KoPrRo}. Similarly, we can
prove that there is a constant $c$ such that, with high probability,
$G(n,1/2)$ with $n \geq 2^{ck\log^2 k}$ satisfies that every graph
on $k$ vertices occurs as an induced monochromatic copy in all
$2$-edge-colorings of $G$.

Very little is known about lower bounds for induced Ramsey numbers
beyond the fact that an induced Ramsey number is at least its
corresponding Ramsey number. A well-known conjecture of Burr and
Erd\H{o}s \cite{BuEr} from 1973 states that for each positive
integer $d$ there is a constant $c(d)$ such that the Ramsey number
$r(H)$ is at most $c(d)k$ for every $d$-degenerate graph $H$ on $k$
vertices. As mentioned earlier, Haxell et al. \cite{HaKoLu} proved
that the induced Ramsey number for the cycle on $k$
vertices is linear in $k$. This implies that the induced
Ramsey number for the path on $k$ vertices is also linear in $k$.
Also, using a star with $2k-1$ edges, it is trivial to see that the induced Ramsey number of a star
with $k$ edges is $2k$. It is natural to ask whether the
Burr-Erd\H{o}s conjecture extends to induced Ramsey numbers. The
following result shows that this fails already for trees, which are
$1$-degenerate graphs.

\begin{theorem} \label{tree} For every $c>0$ and sufficiently large integer $k$
there is  a tree $T$ on $k$ vertices such that
$r_{\textrm{ind}}(T) \geq ck$.
\end{theorem}

The tree $T$ in the above theorem can be taken to be any
sufficiently large tree that contains a matching of linear size and
a star of linear size as subgraphs. It is interesting that the
induced Ramsey number for a path on $k$ vertices or a star on $k$
vertices is linear in $k$, but the induced Ramsey number for a tree
which contains both a path on $k$ vertices and a star on $k$
vertices is superlinear in $k$.

\vspace{0.3cm}
\noindent
{\bf Organization of the paper.}\,\,
In the next section we give short proofs of Theorem
\ref{main} and Theorem \ref{combined} which illustrate our methods.
Section \ref{section12} contains the key lemma that is used as a
replacement for Szemer\'edi's regularity lemma in the proofs of
several results. We answer questions of Chung-Graham and Nikiforov
on the edge distribution in graphs with a forbidden induced
subgraph in Section \ref{sectionmulticolor}. In Section \ref{moreoninduced} we show that any sufficiently
pseudo-random graph of appropriate density has strong induced Ramsey properties.
Combined with known examples of pseudo-random graphs, this leads to explicit
constructions which match and improve the best known estimates for induced Ramsey
numbers. The proof of the result that there are trees whose
induced Ramsey number is superlinear in the number of vertices is in Section \ref{superlinear}. The last section of this
paper contains some concluding remarks together with a discussion of a few conjectures and open problems.
Throughout the paper, we systematically
omit floor and ceiling signs whenever they are not crucial for the
sake of clarity of presentation. We also do not make any serious attempt
to optimize absolute constants in our statements and proofs.

\section{Ramsey-type results for $H$-free graphs}\label{section2}

In this section, we prove Theorems \ref{main} and \ref{combined}.
While we obtain more general results later in the paper, the purpose
of this section is to illustrate on simple examples the main ideas and techniques
that we will use in our proofs. Our theorems
strengthen and generalize results from \cite{Ro} and \cite{PrRo} and the proofs we present
here are shorter and simpler than the original ones.
We start with the proof of Theorem \ref{main}, which uses the following lemma of Erd\H{o}s and
Hajnal \cite{ErHa}. We prove a generalization of this lemma
in Section \ref{sectionmulticolor}.

\begin{lemma}\label{lemmaerdoshajnal}
For each $\epsilon \in (0,1/2)$, graph $H$ on $k$ vertices, and
$H$-free graph $G=(V,E)$ on $n \geq 2$ vertices, there are disjoint
subsets $A$ and $B$ of $V$ with $|A|, |B| \geq
\epsilon^{k-1}\frac{n}{k}$ such that either every vertex in $A$ has
at most $\epsilon|B|$ neighbors in $B$, or every vertex in $A$ has
at least $(1-\epsilon)|B|$ neighbors in $B$.
\end{lemma}

Actually, the statement of the lemma in \cite{ErHa} is a bit weaker
than that of Lemma \ref{lemmaerdoshajnal} but it is easy to get the above statement
by analyzing more carefully the proof of Erd\H{o}s and Hajnal.
Lemma \ref{lemmaerdoshajnal} roughly says that every $H$-free graph
contains two large disjoint vertex subsets such that the edge
density between them is either very small or very large. However, to prove
Theorem \ref{main}, we need to find a large induced subgraph
with such edge density. Our next lemma shows
how one can iterate the bipartite density result of Lemma
\ref{lemmaerdoshajnal} in order to establish the complete density
result of Theorem \ref{main}.

For $\epsilon_1,\epsilon_2 \in (0,1)$ and a graph $H$, define
$\delta(\epsilon_1,\epsilon_2,H)$ to be the largest $\delta$ (which
may be 0) such that for each $H$-free graph on $n$ vertices, there
is an induced subgraph on at least $\delta n$ vertices with edge
density at most $\epsilon_1$ or at least $1-\epsilon_2$. Notice that
for $2 \leq n_0 \leq n_1$, the edge-density of a graph on $n_1$
vertices is the average of the edge-densities of the induced
subgraphs on $n_0$ vertices. Therefore, from definition of $\delta$,
it follows that for every $2 \leq n_0 \leq
\delta(\epsilon_1,\epsilon_2,H) n$ and $H$-free graph $G$ on $n$
vertices, $G$ contains an induced subgraph on exactly $n_0$ vertices
with edge density at most $\epsilon_1$ or at least $1-\epsilon_2$.
Recall that the edge-density $d(A)$ of a subset $A$ of $G$ equals
$e(A)/{|A| \choose 2}$, where $e(A)$ is  the number of edges spanned
by $A$.

\begin{lemma}\label{useful}
Suppose $\epsilon_1,\epsilon_2 \in (0,1)$ with
$\epsilon_1+\epsilon_2<1$ and $H$ is a graph on $k \geq 2$ vertices.
Let $\epsilon=\min(\epsilon_1,\epsilon_2)$. We have
$$\delta(\epsilon_1,\epsilon_2,H) \geq
(\epsilon/4)^{k}k^{-1}\min\Big(\delta\big(3\epsilon_1/2,\epsilon_2,H\big),
\delta\big(\epsilon_1, 3\epsilon_2/2,H\big)\Big).$$
\end{lemma}

\begin{proof}
Let $G$ be a $H$-free graph on $n \geq 2$ vertices. If $n<k$ then we
may consider any two-vertex induced subgraph of $G$ which has always
density either 0 or 1. Therefore, for $G$ of order less than $k$ we
can take $\delta=2/k$, which is clearly larger than the right hand
side of the inequality in the assertion of the lemma. Thus we can
assume that $n \geq k$. Applying Lemma \ref{lemmaerdoshajnal} to $G$
with $\epsilon/4$ in place of $\epsilon$,  we find two subsets $A$
and $B$ with $|A|,|B| \geq (\epsilon/4)^{k-1} n/k $, such that
either every vertex in $A$ is adjacent to at most
$\frac{\epsilon}{4}|B|$ vertices of $B$ or every vertex of $A$ is
adjacent to at least $(1-\frac{\epsilon}{4})|B|$ vertices of $B$.

Consider the first case in which every vertex in $A$ is adjacent to
at most $\frac{\epsilon}{4}|B|$ vertices of $B$ (the other case can
be treated similarly) and let $G[A]$ be the subgraph of $G$ induced
by the set $A$. By definition of function $\delta$, $G[A]$ contains
a subset $A'$ with
$$|A'|= \delta(3\epsilon_1/2,\epsilon_2,H)\Big(\frac{\epsilon}{4}\Big)^{k}\frac{n}{k} \leq
 \delta(3\epsilon_1/2,\epsilon_2,H)|A|,$$
such that the subgraph induced by $A'$ has edge density at most $\frac{3}{2}\epsilon_1$ or at least
$1-\epsilon_2$. If $A'$ has edge density at least $1-\epsilon_2$ we are done, since
$G[A']$ is an induced subgraph of $G$ with at least
$(\epsilon/4)^{k}k^{-1}\delta(3\epsilon_1/2,\epsilon_2,H)n$
vertices and edge density at least $1-\epsilon_2$. So we may assume
that the edge density in $A'$ is at most $\frac{3}{2}\epsilon_1$.

Let $B_1 \subset B$ be those vertices of $B$ that have at most
$\frac{\epsilon}{2}|A'|$ neighbors in $A'$. Since $A' \subset A$, each vertex of
$A'$ has at most $\frac{\epsilon}{4}|B|$ neighbors in $B$ and the number of edges
$e(A',B) \leq \frac{\epsilon}{4}|A'||B|$. Therefore $B_1$ has at least $|B|/2$ vertices.
Then, by definition of $\delta$, $B_1$ contains a
subset $B'$ with
$$|B'| = \delta(3\epsilon_1/2,\epsilon_2,H)\Big(\frac{\epsilon}{4}\Big)^{k}\frac{n}{k}
\leq \delta(3\epsilon_1/2,\epsilon_2,H)|B_1|,$$
such that the induced subgraph $G[B']$ has edge density at most
$\frac{3}{2}\epsilon_1$ or at least $1-\epsilon_2$. If it has edge
density at least $1-\epsilon_2$ we are done, so we may assume
that the edge density $d(B')$ is at most $\frac{3}{2}\epsilon_1$.

Finally to complete the proof note that, since $|A'|=|B'|$, $|A' \cup B'|=2|A'|$,
$d(A'),d(B') \leq \frac{3}{2}\epsilon_1$, and
$d(A',B') \leq \frac{\epsilon_1}{2}$, we have that
\begin{eqnarray*}
e(A' \cup B')&=&e(A')+e(B')+e(A',B') \leq \frac{3}{2}\epsilon_1 {|A'| \choose 2}+
\frac{3}{2}\epsilon_1 {|B'| \choose 2}+
\frac{\epsilon_1}{2}|A'||B'|\\
&=&2\epsilon_1|A'|^2-3\epsilon_1 |A'|/2 \leq \epsilon_1{2|A'| \choose 2}.
\end{eqnarray*}
Therefore, $d(A' \cup B') \leq \epsilon_1$.
\end{proof}

From this lemma, the proof of our first result, that every $H$-free
graph on $n$ vertices contains a subset of at least $2^{-ck(\log
\frac{1}{\epsilon})^2}n$ vertices with edge density either $ \leq
\epsilon$ or $\geq 1-\epsilon$, follows in a few lines.

\noindent {\bf Proof of Theorem \ref{main}:} Notice that if
$\epsilon_1+\epsilon_2 \geq 1$, then trivially
$\delta(\epsilon_1,\epsilon_2,H) =1$. In particular, if
$\epsilon_1\epsilon_2 \geq \frac{1}{4}$, then $\epsilon_1+\epsilon_2
\geq 1$ and $\delta(\epsilon_1,\epsilon_2,H) =1$. Therefore, by
iterating Lemma \ref{useful} for $t=\log \frac{1}{\epsilon^2}/\log
\frac{3}{2}$ iterations and using that $\epsilon \leq 1/2$, we
obtain
$$\delta(\epsilon,\epsilon,H) \geq
\left(\frac{\epsilon^{k}}{4^{k}k}\right)^t \geq 2^{-\frac{2}{\log
3/2}\left(k(\log 1/\epsilon)^2+\left(2k+\log k \right)\log
1/\epsilon\right)} \geq 2^{-15k (\log 1/\epsilon)^2},$$
which, by definition of $\delta$, completes the proof of the theorem. \qed

Recall the Erd\H{o}s-Szemer\'edi theorem, which states that there is
an absolute constant $c$ such that every graph $G$ on $n$ vertices
with edge density $\epsilon \in (0,1/2)$ has a homogeneous set of
size at least $\frac{c \log n}{\epsilon \log \frac{1}{\epsilon}}$.
Theorem \ref{combined} follows from a simple application of Theorem
\ref{main} and the Erd\H{o}s-Szemer\'edi theorem.

\noindent {\bf Proof of Theorem \ref{combined}:} Let $G$ be a graph
on $n$ vertices which is not $k$-universal, i.e., it is $H$-free for
some fixed graph $H$ on $k$ vertices. Fix
$\epsilon=2^{-\frac{1}{5}\sqrt{\frac{\log n}{k}}}$ and apply Theorem
\ref{main} to $G$. It implies that $G$ contains a subset $W \subset
V(G)$ of size at least $2^{-15k(\log
\frac{1}{\epsilon})^2}n=n^{2/5}$ such that the subgraph induced by
$W$ has edge density at most $\epsilon$ or at least $1-\epsilon$.
Applying the Erd\H{o}s-Szemer\'edi theorem to the induced subgraph
$G[W]$ or its complement and using that $\epsilon \log 1/\epsilon
\leq 4\epsilon^{1/2}$ for all $\epsilon \leq 1$, we obtain  a
homogeneous subset $W' \subset W$ with $$|W'| \geq \frac{c \log
n^{2/5}}{\epsilon \log \frac{1}{\epsilon}} \geq \frac{c\log
n}{10\epsilon^{1/2}} \geq
\frac{c}{10}2^{\frac{1}{10}\sqrt{\frac{\log n}{k}}}\log n,$$ which
completes the proof of Theorem \ref{combined}. \qed

\section{Key Lemma}\label{section12}

In this section we present our key lemma. We use it as a
replacement for Szemer\'edi's regularity lemma in the proofs of several Ramsey-type results,
thereby giving much better estimates.  A very special case of this statement was
essentially proved in Lemma \ref{useful} in the previous section.
Our key lemma generalizes the result of Graham, R\"odl, and
Rucinski \cite{GrRoRu} and has a simpler proof
than the one in  \cite{GrRoRu}.
Roughly, our result says that if $(G_1,\ldots,G_r)$ is a
sequence of graphs on the same vertex set $V$ with the property that
every large subset of $V$ contains a pair of large disjoint sets
with small edge density between them in at least one of the graphs $G_i$,
then every large subset of $V$ contains a large set with
small edge density in one of the $G_i$. To formalize this concept,
we need a couple definitions.

For a graph $G=(V,E)$ and disjoint subsets $W_1,\ldots,W_t \subset V$, the {\it density} $d_{G}(W_1,\ldots,W_t)$
between the $t \geq 2$ vertex subsets $W_1,\ldots,W_t$ is defined by
$$d_G(W_1,\ldots,W_t)=\frac{\sum_{ i < j}
e(W_i,W_j)}{\sum_{i < j } |W_i||W_j|}.$$ If
$|W_1|=\ldots=|W_t|$, then
$$d_G(W_1,\ldots,W_t)={t \choose 2}^{-1}\sum_{ i < j }
d_G(W_i,W_j).$$
Also, in this section if $t=1$ we define the density to be zero.

\begin{definition} \label{d31}
For $\alpha,\rho,\epsilon \in [0,1]$ and positive integer $t$, a
sequence $(G_1,\ldots,G_r)$ of graphs on the same vertex set $V$ is
{\bf $(\alpha,\rho,\epsilon,t)$-sparse} if for all subsets $U
\subset V$ with $|U| \geq \alpha |V|$, there are positive integers
$t_1,\ldots,t_r$ such that $\prod_{i=1}^r t_i \geq t$ and for each
$i \in [r]=\{1, \ldots, r\}$ there are disjoint subsets
$W_{i,1},\ldots,W_{i,t_i} \subset U$ with $|W_{i,1}|=\ldots
=|W_{i,t_i}|=\lceil \rho |U|\rceil$ and
$d_{G_i}(W_{i,1},\ldots,W_{i,t_i})\leq \epsilon$.
\end{definition}

We call a graph $(\alpha,\rho,\epsilon,t)$-sparse if the one-term
sequence consisting of that graph is
$(\alpha,\rho,\epsilon,t)$-sparse. By averaging, if $\alpha'\geq
\alpha$, $\rho' \leq \rho$, $\epsilon'\geq \epsilon$, $t' \leq t$,
and $(G_1,\ldots,G_r)$ is $(\alpha,\rho,\epsilon,t)$-sparse, then
$(G_1,\ldots,G_r)$ is also $(\alpha',\rho',\epsilon',t')$-sparse.
The following is our main result in this section.

\begin{lemma}
\label{l32}
If a sequence of graphs $(G_1,\ldots,G_r)$ with common vertex set $V$ is
$(\frac{1}{2}\alpha\rho,\rho',\epsilon,t)$-sparse
and $(\alpha,\rho,\epsilon/4,2)$-sparse,
then $(G_1,\ldots,G_r)$ is also $(\alpha,\frac{1}{2}\rho\rho',\epsilon,2t)$-sparse.
\end{lemma}
\begin{proof}
Since $(G_1,\ldots,G_r)$ is $(\alpha,\rho,\epsilon/4,2)$-sparse,
then for each $U \subset V$ with $|U| \geq \alpha |V|$, there is $i
\in [r]$ and disjoint subsets $X,Y \subset U$ with $|X|=|Y| =
\rho|U|$ and $d_{G_{i}}(X,Y) \leq \epsilon/4$. Let $X_1$ be the set of vertices in
$X$ that have at most $\frac{\epsilon}{2}|Y|$ neighbors in $Y$ in graph $G_i$. Then
$e_{G_i}(X\setminus X_1,Y) \geq \epsilon|X\setminus X_1||Y|/2$ and we also have $e_{G_i}(X,Y) \leq \epsilon|X||Y|/4$.
Therefore $|X_1| \geq |X|/2 \geq \frac{1}{2}\rho|U|$ and by removing extra vertices we
assume that $|X_1|= \frac{1}{2}\rho|U|$.

Since $(G_1,\ldots,G_r)$ is
$(\frac{1}{2}\alpha\rho,\rho',\epsilon,t)$-sparse, then there are
positive integers $t_{1},\ldots,t_{r}$ such that $\prod_{j=1}^r
t_{j} \geq t$ and for each $j \in [r]$ there are disjoint subsets
$X_{j,1},\ldots,X_{j,t_j} \subset X_1$ of size
$|X_{j,1}|=\ldots=|X_{j,t_j}|=\rho' |X_1|$ with density
$d_{G_j}(X_{j,1},\ldots,X_{j,t_j})\leq \epsilon$. Let $Y_1$ the set
of vertices in $Y$ that have at most $\epsilon|X_{i,1} \cup \ldots
\cup X_{i,t_i}|$ neighbors in $X_{i,1} \cup \ldots \cup X_{i,t_i}$
in graph $G_i$. Since every vertex of $X_1$ is adjacent to at most
$\frac{\epsilon}{2}|Y|$ vertices of $Y$ and since $X_{i,1} \cup
\ldots \cup X_{i,t_i}\subset X_1$ we have that $d_{G_i}(X_{i,1} \cup
\ldots \cup X_{i,t_i},Y)\leq \epsilon/2$. On the other hand,
$d_{G_i}(X_{i,1} \cup \ldots \cup X_{i,t_i},Y\setminus Y_1)\geq
\epsilon$. Therefore $|Y_1|\geq |Y|/2$, so again we can assume that
$|Y_1|=\frac{1}{2}\rho|U|=|X_1|$. Since $(G_1,\ldots,G_r)$ is
$(\frac{1}{2}\alpha\rho,\rho',\epsilon,t)$-sparse, then there are
positive integers $s_{1},\ldots,s_{r}$ such that $\prod_{j=1}^r
s_{j} \geq t$ and for each $j \in [r]$ there are disjoint subsets
$Y_{j,1},\ldots,Y_{j,s_j} \subset Y_1$ with
$d_{G_j}(Y_{j,1},\ldots,Y_{j,s_j})\leq \epsilon$ and
$|Y_{j,1}|=\ldots=|Y_{j,s_j}|=\rho' |Y_1|$.

By the above construction, the edge density between
$X_{i,1}\cup \ldots \cup X_{i,t_i}$ and $Y_{i,1}\cup \ldots \cup Y_{i,s_i}$
is bounded from above by $\epsilon$. We also have that both
 $d_{G_i}(X_{i,1},\ldots,X_{i,t_i})$ and $d_{G_i}(Y_{i,1},\ldots,Y_{i,s_i})$ are at most $\epsilon$
and these two sets have the same size.
Therefore $d_{G_i}(X_{i,1},\ldots,X_{i,t_i},Y_{i,1},\ldots,Y_{i,s_i}) \leq \epsilon$, implying
that $(G_1,\ldots,G_r)$ is
$\left(\alpha,\frac{1}{2}\rho\rho',\epsilon,u\right)$-sparse with
$u=(t_i+s_i)\prod_{j \in [r]\setminus \{i\}}\max(t_j,s_j)$ for some $i$.
By the arithmetic mean-geometric mean inequality, we have
$$t^2 \leq \prod_{j=1}^r t_j \prod_{j=1}^r s_j \leq
t_is_i\left(\prod_{j \in [r]\setminus \{i\}} \max(t_j,s_j)\right)^2
=\frac{t_is_i}{(t_i+s_i)^2}u^2 \leq \frac{u^2}{4}.$$
Thus $u \geq 2t$. Altogether this shows that $(G_1,\ldots,G_r)$ is
$\left(\alpha,\frac{1}{2}\rho\rho',\epsilon,2t\right)$-sparse, completing the proof.
\end{proof}

Rather than using this lemma directly, in applications we usually need the following two
corollaries. The first one is obtained by simply applying Lemma \ref{l32} $h-1$ times.

\begin{corollary}\label{secondcor}
If $(G_1,\ldots,G_r)$ is $(\alpha,\rho,\epsilon/4,2)$-sparse and $h$
is a positive integer, then $(G_1,\ldots,G_r)$ is also
$\left((\frac{2}{\rho})^{h-1}\alpha,2^{1-h}\rho^{h},\epsilon,2^h\right)$-sparse.
\end{corollary}

If we use the last statement with $h=r\log \frac{1}{\epsilon}$ and $\alpha=(\frac{\rho}{2})^{h-1}$,
then we get that there is an index $i \in [r]$ and disjoint
subsets $W_1,\ldots,W_{t} \subset V$ with $t \geq 2^{h/r}=
\frac{1}{\epsilon}$, $|W_1|=\ldots=|W_t|= 2^{1-h}\rho^{h}|V|$,
and $d_{G_i}(W_1,\ldots,W_t) \leq \epsilon$. Since
${|W_1| \choose 2} \leq \frac{\epsilon}{t}{t|W_1| \choose 2}$, even if every $W_i$ has edge density one, still the edge density in the
set $W_1 \cup \ldots \cup W_t$ is at most $2\epsilon$. Therefore, (using $\epsilon/2$ instead of $\epsilon$) we
have the following corollary.

\begin{corollary}\label{corbip}
If $(G_1,\ldots,G_r)$ is
$((\frac{\rho}{2})^{h-1},\rho,\epsilon/8,2)$-sparse where $h=r\log \frac{2}{\epsilon}$,
then there is $i \in [r]$ and an
induced subgraph $G'$ of $G_i$ on
$2\epsilon^{-1}2^{1-h}\rho^{h}|V|$ vertices that has edge density at
most $\epsilon$.
\end{corollary}

The key lemma in the paper of Graham, R\"odl, and Rucinski
\cite{GrRoRu} on the Ramsey number of graphs (their Lemma 1) is
essentially the $r=1$ case of Corollary \ref{corbip}.

\section{Edge distribution in $H$-free graphs}\label{sectionmulticolor}

In this section, we obtain several results on the edge distribution of
graphs with a forbidden induced subgraph which answer open questions by
Nikiforov and Chung-Graham. We first prove a strengthening of R\"odl's
theorem (mentioned in the introduction) without using the regularity lemma.
Then we present a proof of Theorem \ref{dev} on the
dependence of error terms in quasirandom properties. We
conclude this section with an upper bound on the maximum edge discrepancy
in subgraphs of $H$-free graphs.
To obtain these results we need the following generalization of Lemma
\ref{lemmaerdoshajnal}.

\begin{lemma}\label{lemmaerdoshajnal4}
Let $H$ be a $k$-vertex graph and let $G$ be a graph on $n \geq k^2$ vertices
that contains less than $n^{k}(1-\frac{k^2}{2n})\prod_{i=1}^{k-1} (1-\delta_i)\epsilon_i^{k-i}$
labeled induced copies of $H$, where $\epsilon_0=1$ and $\epsilon_i,\delta_{i} \in (0,1)$ for all
$1 \leq i \leq k-1$. Then there is an index $i\leq k-1$ and disjoint subsets $A$ and $B$ of $G$ with
$|A|\geq \frac{\delta_in}{k(k-i)}\prod_{j<i} \epsilon_j$ and $|B|
\geq \frac{n}{k}\prod_{j < i} \epsilon_j$ such that either every vertex of $A$
is adjacent to at most $\epsilon_i|B|$ vertices of $B$ or every
vertex of $A$ is adjacent to at least $(1-\epsilon_i)|B|$ vertices
of $B$.
\end{lemma}

\begin{proof}
Let $M$ denote the number of labeled induced copies of $H$ in $G$, which by our assumption is at most
\begin{eqnarray}
\label{M-bound}
M < n^{k}\left(1-\frac{k^2}{2n}\right)\prod_{i=1}^{k-1} (1-\delta_i)\epsilon_i^{k-i}.
\end{eqnarray}
We may assume that the vertex set of $H$ is $[k]$. Consider a random
partition $V_1 \cup \ldots \cup V_k$ of the vertices of $G$ such that each $V_i$ has
cardinality $n/k$. Note that for any such partition there are $(n/k)^k$ ordered
$k$-tuples of vertices of $G$ with the property that the $i$-th vertex of the $k$-tuple is in $V_i$ for all $i
\in [k]$.
On the other hand the total number of ordered $k$-tuples of vertices is $n(n-1)\cdots(n-k+1)$ and each of these
$k$-tuples
has the above property with equal probability. This implies that for any given
$k$-tuple the probability that its $i$-th vertex is in $V_i$ for all $i \in [k]$ equals
$\prod_{i=1}^k \frac{n/k}{n-i+1}$. In particular, by linearity of expectation,
the expected number of labeled induced copies of $H$ in $G$
for which the image of every vertex $i \in [k]$ is in $V_i$ is at most $M\cdot \prod_{i=1}^k \frac{n/k}{n-i+1}$.
Using that $\prod (1-x_i) \geq 1-\sum x_i$ for any $0 \leq x_i \leq 1$ and that $n \geq k^2$, we obtain

\begin{eqnarray*}
\prod_{i=1}^k
\frac{n/k}{n-i+1} &=& k^{-k}\prod_{i=0}^{k-1}(1-i/n)^{-1} \leq
k^{-k}\left(1-\sum_{i=0}^{k-1} i/n\right)^{-1} = k^{-k}\left(1-{k
\choose 2}/n\right)^{-1}\\ &<&
\left(1-\frac{k^2}{2n}\right)^{-1}k^{-k}.
\end{eqnarray*}
This, together with (\ref{M-bound}), shows that there is a partition $V_1 \cup \ldots \cup V_k$ of $G$ into sets
of cardinality $n/k$ such that the total number of labeled induced copies of $H$ in $G$
for which the image of every vertex $i\in [k]$ is in $V_i$ is less than
\begin{equation}
\label{upbound} M\left(1-\frac{k^2}{2n}\right)^{-1}k^{-k}<k^{-k}n^{k}\prod_{i=1}^{k-1}
(1-\delta_i)\epsilon_i^{k-i}.
\end{equation}

We use this estimate to construct sets $A$ and $B$ which satisfy the assertion of the lemma.
For a vertex $v \in V$, the {\it neighborhood} $N(v)$ is the set of
vertices of $G$ that are adjacent to $v$. For $v \in V_i$ and a
subset $S \subset V_j$ with $i \not = j$, let $\tilde{N}(v,S)=N(v) \cap S$
if $(i,j)$ is an edge of $H$ and $\tilde{N}(v,S)=S \setminus N(v)$
otherwise. We will try iteratively to build many induced copies of $H$.
After $i$ steps, we will
have vertices $v_1,\ldots,v_{i}$ with $v_j \in
V_j$ for $j \leq i$ and subsets
$V_{i+1,i}, V_{i+2,i},\ldots, V_{k,i}$ such that
\begin{enumerate}
\item $V_{\ell,i}$ is a subset of $V_{\ell}$ of size $|V_{\ell,i}| \geq \frac{n}{k}\prod_{j=1}^i\epsilon_j$ for all $i+1 \leq
\ell \leq k$,
\item for $1 \leq j<\ell \leq i$, $(v_j,v_{\ell})$ is an edge of $G$  if and
only if $(j,\ell)$ is an edge of $H$,
\item and if $j \leq i<\ell$ and $w \in V_{\ell,i}$, then
$(v_j,w)$ is an edge of $G$ if and only if $(j,\ell)$ is an edge of
$H$.
\end{enumerate}

In the first step, we call a vertex $v\in V_1$ {\it good} if
$|\tilde{N}(v,V_i)| \geq \epsilon_1|V_i|$ for each $i >1$. If less than a
fraction $1-\delta_1$ of the vertices in $V_1$ are good, then, by
the pigeonhole principle, there is a subset $A \subset V_1$ with
$|A|\geq \frac{\delta_1}{k-1}|V_1|=\frac{\delta_1}{k(k-1)}n$ and an index $j > 1$ such that
$|\tilde{N}(v,V_j)| < \epsilon_1|V_j|$ for each $v \in A$. Letting $B=V_j$, one can easily check that
$A$ and $B$ satisfy the assertion of the lemma. Hence, we may assume that at least a
fraction $1-\delta_1$ of the vertices $v_1 \in V_1$ are good, choose any good $v_1$ and
define $V_{i,1}=\tilde{N}(v_1,V_i)$ for $i>1$, completing the first step.

Suppose that after step $i$ the properties 1-3 are satisfied. Then,
in step $i+1$, we again call a vertex $v \in V_{i+1,i}$ {\it good} if
$|\tilde{N}(v,V_{j,i})| \geq \epsilon_{i+1}|V_{j,i}|$ for each $j>i+1$. If
less than a fraction $1-\delta_{i+1}$ vertices of $V_{i+1,i}$ are
good, then, by the pigeonhole principle, there is a subset $A
\subset V_{i+1,i}$ with
$|A|\geq \frac{\delta_{i+1}}{k-i-1}|V_{i+1,i}|$
and index $j >i+1$ such that
$|\tilde{N}(v,V_{j,i})|<\epsilon_{i+1}|V_{j,i}|$ for each $v \in A$. Letting,
$B=V_{j,i}$, one can check using properties 1-3, that
$A$ and $B$ satisfy the assertion of the lemma. Hence, we may assume that a
fraction $1-\delta_{i+1}$ of the vertices $v_{i+1} \in V_{i+1,i}$
are good, choose any good $v_{i+1}$ and define $V_{j,i+1}=\tilde{N}(v_{i+1},V_{j,i})$ for $j>i+2$,
completing step $i+1$. Notice that after step $i+1$, we have
$|V_{j,i+1}|\geq \epsilon_{i+1}|V_{j,i}|$ for $j>i+1$, which guarantees that
property 1 is satisfied. The remaining properties (2 and 3) follow from
our construction of sets $V_{j,i+1}$.

Thus if our process fails in one of the first $k-1$ steps we obtain desired sets $A$ and $B$.
Suppose now that we successfully performed $k-1$ steps. Note that in step $i+1$, we had at least $(1-\delta_{i+1})|V_{i+1,i}| \geq
\frac{n}{k}(1-\delta_{i+1})\prod_{j=1}^i \epsilon_j$ vertices to choose for vertex $v_{i+1}$. Also note that, by property 3, after step $k-1$
we can choose any vertex in the set $V_{k,k-1}$ to be $v_k$. Moreover, by the property 2, every choice of the vertices $v_1, \ldots, v_k$ form a
labeled induced copy of $H$. Altogether, this gives at least
\begin{eqnarray*}
|V_{k,k-1}| \cdot \prod_{i=1}^{k-1}\bigg( \frac{n}{k}(1-\delta_{i})\prod_{0 \leq j<i}
\epsilon_j\bigg) &\geq& \frac{n}{k}\prod_{j=1}^{k-1}\epsilon_j \, \cdot \,
\prod_{i=1}^{k-1}\bigg( \frac{n}{k}(1-\delta_{i})\prod_{0 \leq j<i}
\epsilon_j\bigg)\\
&=&(n/k)^k\prod_{i=1}^{k-1}(1-\delta_i)\epsilon_i^{k-i}
\end{eqnarray*}
labeled induced copies of $H$ for which the image of every vertex $i \in [k]$ is in $V_i$.
This contradicts (\ref{upbound}) and completes the proof.
\end{proof}

Notice that the number of induced copies of $H$ in any induced
subgraph of $G$ is at most the number of induced copies of $H$ in
$G$. Let $\epsilon_i=\epsilon\leq 1/2$ and $\delta_i=\frac{1}{2}$
for $1 \leq i \leq k-1$ and let $\alpha \geq k^2/n$. Applying Lemma
\ref{lemmaerdoshajnal4} with these $\epsilon_i, \delta_i$ to subsets
of $G$ of size $\alpha n$ and using that $\frac{\epsilon^{i-1}}{k-i}
\geq \epsilon^{k-1}, 1-\frac{k^2}{2\alpha n}\geq 1/2$ we obtain the
following corollary.

\begin{corollary}\label{last6}
Let $H$ be a graph with $k$ vertices, $\alpha \geq k^2/n$, $\epsilon
\leq 1/2$, and $G$ be a graph with at most $2^{-k}\epsilon^{k
\choose 2}(\alpha n)^{k}$ induced copies of $H$. Then the pair $(G,
\bar G)$ is $(\alpha,\frac{\epsilon^{k-1}}{2k},\epsilon,2)$-sparse.
\end{corollary}

The next statement strengthens Theorem
\ref{main} by allowing for many induced copies of $H$.
It follows from Corollary \ref{corbip} with $r=2, h=2\log (2/\epsilon),
\rho=\frac{\epsilon^{k-1}}{2k}$,
combined with the last statement in which we set $\alpha=(\rho/2)^{h-1}$.

\begin{corollary}\label{maingeneralized7}
There is a constant $c$ such that for each $\epsilon \in (0,1/2)$
and graph $H$ on $k$ vertices, every graph $G$ on $n$ vertices with less than
$2^{-c(k \log \frac{1}{\epsilon})^2}n^k$ induced copies of
$H$ contains an induced subgraph of size at least $2^{-ck (\log
\frac{1}{\epsilon})^2}n$ with edge density at most $\epsilon$ or at
least $1-\epsilon$.
\end{corollary}

This result demonstrates that for each $\epsilon \in (0,1/2)$ and
graph $H$, there exist positive constants
$\delta^*=\delta^*(\epsilon,H)$ and $\kappa^*=\kappa^*(\epsilon,H)$
such that every graph $G=(V,E)$ on $n$ vertices with less than
$\kappa^* \,n^k$ induced copies of $H$ contains a subset $W \subset
V$ of size at least $\delta^*\,n$ such that the edge density of $W$
is at most $\epsilon$ or at least $1-\epsilon$. Furthermore, there
is a constant $c$ such that we can take $\delta^*(\epsilon,H)=2^{-ck
(\log \frac{1}{\epsilon})^2}$ and $\kappa^*(\epsilon,H)=2^{-c(k\log
\frac{1}{\epsilon})^2}$. Applying Corollary \ref{maingeneralized7}
recursively one can obtain an equitable partition of $G$ into a
small number of subsets each with low or high density.

\begin{theorem}\label{weakversion}
For each $\epsilon \in (0,1/2)$ and graph $H$ on $k$ vertices, there
are positive constants $\kappa=\kappa(\epsilon,H)$ and
$C=C(\epsilon,H)$ such that every graph $G=(V,E)$ on $n$ vertices
with less than $\kappa\, n^{k}$ induced copies of $H$, there is an
equitable partition $V=\bigcup_{i=1}^{\ell} V_i$ such that $\ell
\leq C$ and the edge density in each $V_i$  is at most $\epsilon$ or
at least $1-\epsilon$.
\end{theorem}

\noindent This extension of R\"odl's theorem was proved by Nikiforov
\cite{Ni} using the regularity lemma and therefore it had quite poor
(tower like) dependence of $\kappa$ and $C$ on $\epsilon$ and $k$.
Obtaining a proof without using the regularity lemma was the main
open problem raised in \cite{Ni} .

\vspace{0.2cm}
\noindent
{\bf Proof of Theorem \ref{weakversion}.}\,
Let
$\kappa(\epsilon,H)=(\frac{\epsilon}{4})^k\kappa^*(\frac{\epsilon}{4},H)$
and $C(\epsilon,H)=4/(\epsilon \delta^*(\frac{\epsilon}{4},H))$, where
$\kappa^*$ and $\delta^*$ were defined above.
Take a subset $W_1 \subset V$ of size $
\delta^*(\frac{\epsilon}{4},H)\frac{\epsilon}{4}n$ whose edge
density is at most $\frac{\epsilon}{4}$ or at least
$1-\frac{\epsilon}{4}$, and set $U_1 =V\setminus W_1$. For $j \geq
1$, if $|U_j| \geq \frac{\epsilon}{4} n$, then by definition of $\kappa$ we have that the number of induced copies of
$H$ in $U_j$ is at most (the number of such copies in $G$) $\kappa\,n^k=
(\frac{\epsilon}{4})^k\kappa^*\,n^k \leq \kappa^*|U_j|^k$.
Therefore by definition of $\kappa^*$ and $\delta^*$ we can
find a subset $W_{j+1} \subset U_j$ of size
$\delta^*\frac{\epsilon}{4}n \leq \delta^*|U_j|$ whose edge
density is at most $\frac{\epsilon}{4}$ or at least
$1-\frac{\epsilon}{4}$, and set $U_{j+1}=U_j \setminus W_{j+1}$.

Once this process stops, we have disjoint sets
$W_1,\ldots,W_{\ell}$, each with the same cardinality, and a subset
$U_{\ell}$ of cardinality at most $\frac{\epsilon}{4}n$. The number
$\ell$ is at most $$n/|W_1| \leq 4/(\epsilon
\delta^*(\frac{\epsilon}{4},H)).$$

Partition set $U_\ell$ into $\ell$ equal parts $T_1, \ldots, T_\ell$
and let $V_j=W_j \cup T_j$ for $1 \leq j \leq \ell$.
Notice that $V=V_1 \cup \ldots \cup V_{\ell}$ is
an equitable partition of $V$. By definition,
$|T_j|=|U_\ell|/\ell \leq \frac{\epsilon}{4}n/\ell$.
On the other hand $|W_j|=(n-|U_\ell|)/\ell \geq (1-\epsilon/4)n/\ell$.
Since $1-\epsilon/4 >7/8$, this implies that
$$|T_j|\leq \frac{\epsilon}{4}n/\ell \leq \frac{\epsilon}{4}\big(1-\epsilon/4\big)^{-1}|W_j| \leq \frac{2\epsilon}{7}|W_j|.$$
We next look at the edge density in $V_j$. If the edge density in $W_j$ is at most $\epsilon/4$, then
using the above bound on $|T_j|$, it is easy to check that the number of edges in $V_j$ is at most
$${|T_j| \choose 2}+|T_j||W_j|+\frac{\epsilon}{4}{|W_j| \choose 2}
\leq \epsilon{|W_j| \choose 2} \leq \epsilon{|V_j| \choose 2}.$$
Hence, the edge density in each such $V_j$ is at most $\epsilon$.
Similarly, if the edge density in $W_j$ is at least
$1-\frac{\epsilon}{4}$, then the edge density in $V_j$ is at least
$1-\epsilon$. This completes the proof.
\hfill $\Box$

\vspace{0.2cm}
We next use Lemma \ref{lemmaerdoshajnal4} to prove that there is a constant $c>0$ such that
every graph $G$ on $n$ vertices which contains at most
$(1-\epsilon)2^{-{k \choose 2}}n^k$ labeled
induced copies of some fixed $k$-vertex graph $H$ has a subset
$S$ of size $|S|=\lfloor n/2 \rfloor$ with $|e(S)-\frac{n^2}{16}| \geq \epsilon
c^{-k} n^2$.

\vspace{0.25cm}
\noindent
{\bf Proof of Theorem \ref{dev}.}\, For $1 \leq i \leq k-1$, let
$\epsilon_i=\frac{1}{2}(1-2^{i-k-2}\epsilon)$ and
$\delta_i=2^{i-k-2}\epsilon$. Notice that for all $i \leq k-1$
\begin{eqnarray}
\label{eq3}
\prod_{j<i}\epsilon_j=2^{-i+1}\prod_{j<i}\big(1-2^{j-k-2}\epsilon\big)\geq
2^{-i+1}\bigg(1-\epsilon \sum_{j<k-1}2^{j-k-2}\bigg) \geq
2^{-i+1}(1-\epsilon/8)>2^{-i}
\end{eqnarray}
and also that
\begin{eqnarray*}
\prod_{i=1}^{k-1}(1-\delta_i)\epsilon_i^{k-i}&=&
2^{-{k \choose 2}}\prod_{i=1}^{k-1}\big(1-2^{i-k-2}\epsilon \big)^{k-i+1}=
2^{-{k \choose
2}}\prod_{j=2}^{k} \big(1-\epsilon2^{-j-1}\big)^{j}\\ &\geq&
2^{-{k \choose 2}}\bigg(1-\epsilon\sum_{j=2}^k \frac{j}{2^{j+1}}\bigg)
>
\left(1-\frac{\epsilon}{2}\right)2^{-{k \choose 2}}.
\end{eqnarray*}
We may assume that $\epsilon \geq k^2/n$ since otherwise
by choosing constant $c$ large enough we get that $\epsilon c^{-k}n^2<1$ and the conclusion of the
theorem follows easily. Therefore
$$ \left(1-\frac{k^2}{2n}\right)\prod_{i=1}^{k-1} (1-\delta_i)\epsilon_i^{k-i} \geq
\left(1-\frac{\epsilon}{2}\right)^22^{-{k \choose 2}} >
(1-\epsilon)2^{-{k \choose 2}},$$ and we can apply Lemma
\ref{lemmaerdoshajnal4} with $\epsilon_i$ and $\delta_i$ as above to
our graph $G$ since it contains at most $(1-\epsilon)2^{-{k \choose
2}}n^k$ labeled induced copies of $H$.  This lemma, together with
(\ref{eq3}), implies that there is an index $i \leq k-1$ and
disjoint subsets $A$ and $B$ with
$$|A|\geq \frac{\delta_in}{k(k-i)}\prod_{j<i} \epsilon_j \geq k^{-2}2^{-k-2}n,$$ $$|B|
\geq \frac{n}{k}\prod_{j < i} \epsilon_j \geq 2^{-i}k^{-1}n,$$ and
every element of $A$ is adjacent to at most $\epsilon_i|B|$ elements
of $B$ or every element of $A$ is adjacent to at least
$(1-\epsilon_i)|B|$ elements of $B$. In either case, we have
$$\left|e(A,B)-\frac{1}{2}|A||B|\right| \geq
\left(\frac{1}{2}-\epsilon_i\right)|A||B|=2^{i-k-3}\epsilon|A||B| \geq k^{-3}2^{-2k-5}\epsilon n^2.$$
Note that
$$e(A,B)-\frac{1}{2}|A||B|=\left(e(A \cup B)-\frac{1}{2}{|A \cup B|
\choose 2}\right)-\left(e(A)-\frac{1}{2}{|A| \choose
2}\right)-\left(e(B)-\frac{1}{2}{|B| \choose 2}\right).$$ It follows
from the triangle inequality that there is some subset of vertices
$R \in \{A,B, A \cup B\}$ such that
\begin{equation}
\label{subset}
\left|e(R)-\frac{1}{2}{|R| \choose 2}\right| \geq
\frac{1}{3} k^{-3}2^{-2k-5}\epsilon n^2,
\end{equation}
i.e.,  it deviates by at least $\epsilon k^{-3}2^{-2k-5}n^2/3$ edges from having edge density $1/2$.
To finish the proof we will use the lemma of Erd\H{o}s et al. \cite{ErGoPaSp}, mentioned in the introduction.
This lemma says that if graph $G$ on $n$ vertices with edge density $\eta$ has a subset that
deviates by $D$ edges from having edge density $\eta$, then it also
has a subset of size $n/2$ that deviates by at least $D/5$ edges from having edge density $\eta$.
Note that if the edge density of our graph $G$ is either larger than
$1/2+\epsilon k^{-3}2^{-2k-5}n^2/30$ or smaller than  $1/2-\epsilon k^{-3}2^{-2k-5}n^2/30$
than by averaging over all subsets of size $n/2$ we will find subset $S$ satisfying our assertion.
Otherwise, if the edge density $\eta$ of $G$ satisfies $|\eta-1/2| \leq \epsilon k^{-3}2^{-2k-5}n^2/30$, then
the subset $R$ from (\ref{subset}) deviates by at least
$\epsilon k^{-3}2^{-2k-5}n^2/3- \epsilon k^{-3}2^{-2k-5}n^2/30 \geq \epsilon k^{-3}2^{-2k}n^2/4$ edges from
having edge density $\eta$. Then, by the lemma of Erd\H{o}s et al., $G$ has a subset $S$ of cardinality $n/2$
that deviates by at least $\epsilon k^{-3}2^{-2k-5}n^2/20$ edges from having edge density $\eta$.
This $S$ satisfies
$$\left|e(S)-\frac{1}{4}|S|^2\right| \geq \epsilon k^{-3}2^{-2k-5}n^2/20-\epsilon k^{-3}2^{-2k-5}n^2/30=
\Omega\left(\epsilon
k^{-3}2^{-2k}n^2\right),$$  completing the proof. \qed

\vspace{0.1cm}
For positive integers $k$ and $n$, recall that $D(k,n)$ denotes the largest
integer such that every graph $G$ on $n$ vertices that is $H$-free for some $k$-vertex graph
$H$ contains a subset $S$ of size $n/2$
with $|e(S)-\frac{1}{16}n^2|>D(k,n)$.
We end this section by proving the upper bound on $D(k,n)$.

\begin{proposition}
There is a constant $c>0$ such that for all positive integers $k$ and
$n \geq 2^{k/2}$, there is a $K_k$-free graph $G$ on $n$ vertices such that for
every subset $S$ of $n/2$ vertices of $G$,
$$\Big|e(S)-\frac{1}{16}n^2\Big|<c2^{-k/4}n^2.$$
\end{proposition}

\begin{proof}
Consider the random graph $G(\ell,1/2)$ with $\ell=2^{ k /2}$. For
every subset of vertices $X$ in this graph the number of edges in
$X$ is a binomially distributed random variable with expectation
$\frac{|X|(|X|-1)}{4}$. Therefore by Chernoff's bound (see, e.g.,
Appendix A in \cite{AlSp}), the probability that it deviates from
this value by $t$ is at most $2e^{-t^2/|X|^2}$. Thus choosing
$t=1.5\ell^{3/2}$ we obtain that the probability that there is a
subset of vertices $X$ such that
$\big|e(X)-\frac{|X|(|X|-1)}{4}\big|>t$ is at most $2^\ell \cdot
2e^{-t^2/\ell^2} \ll 1$. This implies that there is graph $\Gamma$
on $\ell$ vertices such that every subset $X$ of $\Gamma$ satisfies
\begin{equation}
\label{random}
\Big|e(X)-\frac{1}{4}|X|^2\Big| \leq 2\ell^{3/2}.
\end{equation}

Let $G$ be the graph obtained by replacing every
vertex $u$ of $\Gamma$ with an independent set $I_u$, of size $n/\ell$,
and by replacing every edge $(u,v)$ of $\Gamma$ with a complete bipartite graph, whose
partition classes are independent sets $I_u$ and $I_v$.
Clearly, since $\Gamma$ does not contain $K_k$, then neither does $G$.
We claim that graph $G$ satisfies the assertion of the proposition.
Suppose for contradiction that there is a subset $S$ of  $n/2$ vertices of $G$ satisfying
$$e(S)-\frac{1}{16}n^2 >4\ell^{3/2}(n/\ell)^2=4\ell^{-1/2}n^2=2^{-k/4+2}n^2,$$
(the other case when $e(S)-n^2/16<-4\ell^{-1/2}n^2$ can be treated similarly).
For every vertex $u \in \Gamma$ let the size of $S \cap I_u$ be $a_un/\ell$.
By definition, $0 \leq a_u \leq 1$ and since $S$ has size $n/2$ we have that $\sum_u a_u=\ell/2$.
We also have that
$$e(S)=\sum_{(u,v)\in E(\Gamma)} a_u a_v \cdot (n/\ell)^2>\frac{1}{16}n^2+4\ell^{3/2}(n/\ell)^2,$$
and therefore
$$\sum_{(u,v)\in E(\Gamma)} a_u a_v >\ell^2/16+4\ell^{3/2}=\frac{1}{4}\Big(\sum_u a_u\Big)^2 +4\ell^{3/2}.$$
Consider a random  subset $Y$ of $\Gamma$ obtained by choosing every
vertex $u$ randomly and independently with probability $a_u$. Since
all choices were independent we have that
$$\mathbb{E}\big[|Y|^2\big]= \sum_u a_u +\sum_{u\not =v}a_ua_v \leq
\big(\sum_u a_u\big)^2+\ell/2.$$
We also have that the
expected number of edges spanned by $Y$ is
$\mathbb{E}\big[e(Y)\big]=\sum_{(u,v)\in E(\Gamma)} a_u a_v$. Then, by the above discussion,
$\mathbb{E}\big[e(Y)-|Y|^2/4\big] >3\ell^{3/2}$. In
particular, there is subset $Y$ of $\Gamma$ with this property,
 which contradicts
(\ref{random}). This shows that every subset  $S$ of  $n/2$
vertices of $G$ satisfies
$$\Big|e(S)-\frac{1}{16}n^2\Big| \leq 2^{-k/4+2}n^2$$
and completes the proof.
\end{proof}

\section{Induced Ramsey Numbers and Pseudorandom Graphs} \label{moreoninduced}
The main result in this section is Theorem \ref{quasirandominduced},
which shows that any sufficiently pseudo-random graph of appropriate
density has strong induced Ramsey properties. It generalizes Theorem
\ref{quasicor1} and Corollary \ref{payley} from the introduction.
Combined with known examples of pseudo-random graphs, this theorem
gives various explicit constructions which match and improve the
best known estimates for induced Ramsey numbers.

The idea of the proof of Theorem \ref{quasirandominduced} is rather
simple. We have a sufficiently large, pseudo-random graph $G$ that
is not too sparse or dense. We also have $d$-degenerate graphs $H_1$
and $H_2$ each with vertex set $[k]$ and chromatic number at most
$q$. We suppose for contradiction that
there is a red-blue edge-coloring of $G$ without an induced red
copy of $H_1$ and without an induced blue copy
of $H_2$. We may view the red-blue coloring of $G$ as a
red-blue-green edge-coloring of the complete graph $K_{|G|}$, in
which the edges of $G$ have their original color, and the edges of
the complement $\bar G$ are colored green. The fact that in $G$
there is no induced red copy of $H_1$ means that the red-blue-green
coloring of $K_{|G|}$ does not contain a particular red-green
coloring of the the complete graph $K_k$. Then we prove, similar to
Lemma \ref{lemmaerdoshajnal} of Erd\H{o}s and Hajnal,  that any
large subset of vertices of $G$ contains two large disjoint subsets
for which the edge density in color red between them is small. By
using the key lemma from Section 3, we find $k$ large disjoint
vertex subsets $V_1,\ldots,V_k$ of $G$ for which the edge density in
color red is small between any pair $(V_i,V_j)$ for which $(i,j)$ an
edge of $H_2$.

Next we try to find an induced blue copy of $H_2$ with vertex $i$ in
$V_i$ for all $i \in [k]$. Since the edge density between $V_i$ and
$V_j$ in color red is sufficiently small for every edge $(i,j)$ of
$H_2$, we can build an induced blue copy of $H_2$ one vertex at a
time. At each step of this process we use pseudo-randomness of $G$
to make sure that the existing possible subsets for not yet embedded
vertices of $H_2$ are sufficiently large and that the density of red
edges does not increase a lot between any pair of subsets
corresponding to adjacent vertices of $H_2$. This last part of the
proof, embedding an induced blue copy of $H_2$, is the most
technically involved and handled by Lemma \ref{densitylemma}.

Recall that $[i]=\{1, \ldots, i\}$ and that a graph is $d$-degenerate if every subgraph has a
vertex of degree at most $d$. For an edge-coloring $\Psi:E(K_{k})\rightarrow [r]$, we say that
another edge-coloring $\Phi:E(K_{n})\rightarrow [s]$ is {\it
$\Psi$-free} if, for every subset $W$ of size $k$ of the complete graph $K_n$,
the restriction of $\Phi$ to $W$ is not isomorphic to
$\Psi$. In the following lemma, we have a coloring $\Psi$ of the
edges of the complete graph $K_k$ with colors $1$ and $2$ such that
the graph of color $2$ is $d$-degenerate. We also have a $\Psi$-free coloring
$\Phi$ of the edges of the complete graph $K_n$ such that between any two large subsets of vertices
there are sufficiently many edges of color 1. With these
assumptions, we show that there are two large subsets of $K_n$
which in coloring $\Phi$ have few edges of color $2$ between them. A graph $G$ is {\it
bi-$(\epsilon,\delta)$-dense} if $d(A,B)>\epsilon$ holds for all
disjoint subsets $A,B \subset V(G)$ with $|A|,|B| \geq \delta
|V(G)|$.

\begin{lemma}\label{lemmaerdoshajnal2}
Let $d$ and $k$ be positive integers and $\Psi:E(K_k) \rightarrow
[2]$ be a $2$-coloring of the edges of $K_k$ such that the graph of
color $2$ is $d$-degenerate. Suppose that $q,\epsilon \in (0,1)$ and
$\Phi:E(K_n) \rightarrow [s]$ is a $\Psi$-free edge-coloring such
that the graph of color $1$ is
bi-$(q,\epsilon^dq^{k}k^{-2})$-dense. Then there are disjoint
subsets $A$ and $B$ of $K_n$ with $|A|,|B| \geq
\epsilon^{d}q^{k} k^{-2}n$ such that every vertex of $A$ is connected to at most
$\epsilon|B|$ vertices in $B$ by edges of color 2.
\end{lemma}
\begin{proof}
Note that from definition, the vertices of every $d$-degenerate
graph can be labeled $1,2, \ldots$ such that for every vertex $\ell$
the number of vertices $j<\ell$ adjacent to it is at most $d$.
(Indeed, remove from the graph a vertex of minimum degree, place it
in the end of the list and repeat this process in the remaining
subgraph.) Therefore we may assume that the labeling $1, \ldots, k$
of vertices of $K_k$ has the property that for every $\ell \in [k]$
there are at most $d$ vertices $j <\ell$ such that the color
$\Psi(j,\ell)=2$. Partition the vertices of $K_n$ into sets $V_1
\cup \ldots \cup V_{k}$ each of size $\frac{n}{k}$. For $w \in V_i$
and a subset $S \subset V_j$ with $j \not = i$, let $N(w,S)=\{s \in
S ~|~ \Phi(w,s)=\Psi(i,j)\}.$ For $i<\ell $, let $D(\ell,i)$ denote
the number of vertices $j \leq i$ such that the color
$\Psi(j,\ell)=2$. By the above assumption, $D(\ell,i) \leq d$ for $1
\leq i < \ell \leq k$.

We will try iteratively to build a copy of $K_{k}$ with coloring
$\Psi$. After $i$ steps, we either find two disjoint subsets of vertices $A, B$ which satisfy the assertion of the lemma or we
will have vertices $v_1,\ldots,v_{i}$ and
subsets $V_{i+1,i},V_{i+2,i},\ldots,V_{k,i}$ such that
\begin{enumerate}
\item $V_{\ell,i}$ is a subset of $V_\ell$ of size $|V_{\ell,i}| \geq \epsilon^{D(\ell,i)}q^{i-D(\ell,i)}|V_{\ell}|$ for all $i+1 \leq
\ell \leq k$,
\item $\Phi(v_j,v_{\ell})=\Psi(j,\ell)$ for $1 \leq j<\ell \leq i$,
\item and if $j \leq i< \ell$ and $w \in V_{\ell,i}$, then
$\Phi(v_j,w)=\Psi(j,\ell)$.
\end{enumerate}

In the first step, we call a vertex $w\in V_1$ {\it good} if
$|N(w,V_j)| \geq \epsilon|V_j|$ for all $j>1$ with $\Psi(1,j) =2$
and $|N(w,V_j)| \geq q|V_i|$ for all $j>1$ with $\Psi(1,j)=1$. If
there is no good vertex in $V_1$, then there is a subset $A \subset
V_1$ with $|A|\geq \frac{1}{k-1}|V_1|$ and index $j >1$ such that
either $\Psi(1,j)=1$ and every vertex $w \in A$ has fewer than
$q|V_j|$ edges of color $1$ to $V_j$ or $\Psi(1,j)=2$ and
every vertex $w \in A$ is connected to less than $\epsilon|V_j|$
vertices in $V_j$ by edges of color 2. Letting $B=V_j$, we conclude
that the first case is impossible since the graph of color $1$ is
bi-$(q,\epsilon^dq^{k}k^{-2})$-dense, while in the second case we
would be done, since $A$ and $B$ would satisfy the assertion of the
lemma. Therefore, we may assume that there is a good vertex $v_1 \in
V_1$, and we define $V_{i,1}=N(v_1,V_i)$ for $i>1$.

Suppose that after step $i$ the properties 1-3 are still satisfied.
Then, in step $i+1$, a vertex $w \in V_{i+1,i}$ is called {\it
good} if $|N(w,V_{j,i})| \geq \epsilon|V_{j,i}|$ for each $j> i+1$
with $\Psi(i+1,j) =2$ and $|N(w,V_{j,i})| \geq q|V_{j,i}|$ for each
$j >i+1$ with $\Psi(i+1,j) =1$. If there is no good vertex in
$V_{i+1,i}$, then there is a subset $A \subset V_{i+1,i}$ with
$|A|\geq \frac{1}{k-i-1}|V_{i+1,i}|$ and $j >i+1$ such that either
$\Psi(i+1,j)=1$ and every vertex $w \in A$ has fewer than
$q|V_{j,i}|$ edges of color $1$ to $V_{j,i}$ or
$\Psi(1,j)=2$ and every vertex $w \in A$ is connected to less than
 $\epsilon|V_{j,i}|$ vertices in $V_{j,i}$ by edges of color
2. Note that even in the last step when $i+1=k$ the size of $A$ is
still at least $|V_{k,k-1}|/k\geq \epsilon^dq^k|V_k|/k\geq
\epsilon^dq^kk^{-2}n$. Therefore, letting $B=V_{j,i}$, we conclude
that as before the first case is impossible since the graph of color
$1$ is bi-$(q,\epsilon^dq^{k}k^{-2})$-dense, while the second case
would complete the proof,  since $A$ and $B$ would satisfy the
assertion of the lemma. Hence, we may assume that there is a good
vertex $v_{i+1} \in V_{i+1,i}$, and we define
$V_{j,i+1}=N(v_{i+1},V_{j,i})$ for $j >i+1$.
Note that $|V_{j,i+1}| \geq q|V_{j,i}|$ if $\Psi(i+1,j) =1$
and $|V_{j,i+1}| \geq \epsilon |V_{j,i}|$ if $\Psi(i+1,j) =2$.
This implies that after step $i+1$ we have
that $|V_{\ell,i+1}| \geq
\epsilon^{D(\ell,i+1)}q^{i+1-D(\ell,i+1)}|V_{\ell}|$ for all $i+2
\leq \ell \leq k$.

The iterative process must stop at one of the steps $j \leq k-1$,
since otherwise the coloring $\Phi$ would not be $\Psi$-free.
As we already explained above, when this happens we have two disjoint subsets $A$ and $B$
that satisfy the assertion of the lemma.
\end{proof}

Notice that if coloring $\Phi:K_n \rightarrow [s]$ is $\Psi$-free,
then so is $\Phi$ restricted to any subset of $K_n$ of size $\alpha n$. Therefore,
Lemma \ref{lemmaerdoshajnal2} has the following corollary.

\begin{corollary}\label{last}
Let $d$ and $k$ be positive integers and $\Psi:E(K_k) \rightarrow
[2]$ be a $2$-coloring of the edges of $K_k$ such that the graph of
color $2$ is $d$-degenerate. If $q,\alpha,\epsilon \in (0,1)$ and
$\Phi:E(K_n) \rightarrow [s]$ is a $\Psi$-free edge-coloring such
that the graph of color $1$ is bi-$(q,\alpha\rho)$-dense with
$\rho=\epsilon^{d}q^{k} k^{-2}$, then the graph of color $2$ is
$(\alpha,\rho,\epsilon,2)$-sparse.
\end{corollary}

\noindent
The next statement follows immediately from
Corollary \ref{last} (with $\epsilon/4$ instead of $\epsilon $) and Corollary \ref{secondcor}.

\begin{corollary}\label{last2}
Let $d$, $k$, and $h$ be positive integers and $\Psi:E(K_k)
\rightarrow [2]$ be a $2$-coloring of the edges of $K_k$ such that
the graph of color $2$ is $d$-degenerate. Suppose that
$q,\alpha,\epsilon \in (0,1)$ and $\Phi:E(K_n) \rightarrow [s]$ is
a $\Psi$-free edge-coloring such that the graph of color $1$ is
bi-$(q,\alpha\rho)$-dense with
$\rho=(\epsilon/4)^dq^{k}k^{-2}$. Then the graph of color $2$ is
$((\frac{2}{\rho})^{h-1}\alpha,2^{1-h}\rho^h,\epsilon,2^h)$-sparse.
\end{corollary}

\noindent
Pending one additional lemma, we are now ready to prove the main result of this
section, showing that pseudo-random graphs have strong induced
Ramsey properties.

\begin{theorem}\label{quasirandominduced}
Let $\chi \geq 2$ and $G$ be a $(p,\lambda)$-pseudo-random graph with
$0 <p \leq 3/4$ and $\lambda \leq ((\frac{p}{10k})^d2^{-pk})^{20\log
\chi}n$. Then every $d$-degenerate graph on $k$ vertices with chromatic
number at most $\chi$ occurs as an induced monochromatic copy in every
$2$-coloring of the edges of $G$. Moreover, all of these induced
monochromatic copies can be found in the same color.
\end{theorem}

Taking $p=1/k$, $n=k^{cd\log \chi}$ and constant
$c$ sufficiently large so that $((\frac{p}{10k})^d2^{-pk})^{20\log \chi}>n^{-0.1}$
one can easily see that this result implies Theorem \ref{quasicor1}.
To obtain Corollary \ref{payley}, recall that for a prime power $n$, the Paley graph $P_n$ has vertex
set $\mathbb{F}_n$ and distinct vertices $x,y \in \mathbb{F}_n$ are
adjacent if $x-y$ is a square. This graph is $(1/2,\lambda)$-pseudo-random with $\lambda=\sqrt{n}$
(see e.g., \cite{KrSu}).
Therefore, for sufficiently large constant $c$, the above theorem with $n=2^{ck\log^2 k}$, $p=1/2$ and $d=\chi=k$
implies that every graph on $k$ vertices occurs as an
induced monochromatic copy in all $2$-edge-colorings of the Paley graph.
Similarly, one can prove that there is a constant $c$ such that, with high probability,
the random graph $G(n,1/2)$ with $n \geq 2^{ck\log^2 k}$ satisfies that every graph
on $k$ vertices occurs as an induced monochromatic copy in all
$2$-edge-colorings of $G$.

\vspace{0.1cm} \noindent {\bf Proof of Theorem
\ref{quasirandominduced}.}\, Suppose for contradiction that there is
an edge-coloring $\Phi_0$ of $G$ with colors red and blue, and
$d$-degenerate graphs $H_1$ and $H_2$ each having $k$ vertices and
chromatic number at most $\chi$ such that there is no induced red
copy of $H_1$ and no induced blue copy of $H_2$. Since $H_1, H_2$
are $d$-degenerate graphs on $k$ vertices we may suppose that their
vertex set is $[k]$ and every vertex $i$ has at most $d$ neighbors
less than $i$ in both $H_1$ and $H_2$.

Consider the  red-blue-green edge-coloring $\Phi$ of the complete graph $K_n$, in which the
edges of $G$ have their original coloring $\Phi_0$,
and the edges of the complement $\bar G$ are colored green.
Let $\Psi$ be the edge-coloring of the complete graph $K_k$ where the red edges form a
copy of $H_1$ and the remaining edges are green. By assumption, the
coloring $\Phi$ is $\Psi$-free.
Since $G$ is $(p,\lambda)$-pseudo-random, we have that the density of edges
in $\bar G$ between any two disjoint sets $A, B$ of size at least $6p^{-1} \lambda$ is at least
$$d_{\bar G}(A,B)=1-d_G(A,B)\geq 1-\Big(p+\frac{\lambda}{\sqrt{|A||B|}}\Big) \geq 1-\frac{7}{6}p.$$
Therefore the green graph in coloring $\Phi$ is
bi-$(q,6p^{-1}\frac{\lambda}{n})$-dense for $q=1-7p/6$.

Let $\epsilon= \frac{p}{1000k^6}$,
$\rho=(\epsilon/4)^dq^{k}k^{-2}$, $h=\log \chi $, and
$\alpha=(\rho/2)^{h-1} $. Using that
$q=1-7p/6$ and $\lambda/n \leq ((\frac{p}{10k})^d2^{-pk})^{20\log \chi}$ it is straightforward to check that
$6p^{-1}\frac{\lambda}{n} \leq
2^{1-h}\rho^{h}=\alpha \rho$. By Corollary \ref{last2} and Definition \ref{d31}, there are
$2^h=\chi$ subsets $W_1,\ldots,W_{\chi}$ of $K_n$ with $|W_1|=\ldots=|W_\chi| \geq
2^{1-h}\rho^{h}n$, such that the sum of densities of red edges
between all pairs $W_i$ and $W_j$ is at most ${\chi \choose 2}\epsilon$.
Hence, the density between $W_i$
and $W_j$ is also at most $\chi^2\epsilon$ for all $1 \leq i < j \leq \chi$.
Partition every set $W_i$ into $k$ subsets
each of size $|W_i|/k \geq \frac{1}{k}2^{1-h}\rho^{h}n$.
Since the chromatic number of $H_2$ is at most $\chi$ and it has $k$ vertices, we can
choose for every vertex $i$ of $H_2$ one of these subsets, which we call $V_i$, such that all subsets
corresponding to vertices of $H_2$ in the same color class (of a proper $\chi$-coloring)
come from the same set $W_\ell$.
In particular, for every edge $(i,j)$ of $H_2$, the corresponding sets
$V_i$ and $V_j$ lie in two different sets $\{W_\ell\}$. Since the size of $V_i$'s is by
a factor $k$ smaller than the size
of $W_\ell$'s the density of red edges between $V_i$ and $V_j$
corresponding to an edge in $H_2$ is at most $k^2\chi^2\epsilon \leq
\frac{p}{1000k^2}$ (note that it can increase by a factor at most $k^2$ compare to density between sets $\{W_\ell\}$).
Notice that the subgraph $G' \subset G$ induced by $V_1 \cup \ldots \cup
V_{k}$ has $n' \geq 2^{1-h}\rho^{h}n$ vertices and is also $(p,\lambda)$-pseudo-random.
By the definitions of $\rho$ and $h$, and our assumption on $\lambda$, we have that
$$\lambda/n'\leq 2^{h-1}\rho^{-h}\lambda/n \leq 2^{h-1}\rho^{-h} \left(\Big(\frac{p}{10k}\Big)^d2^{-pk}\right)^{20\log \chi}
\leq   \left(\Big(\frac{p}{10k}\Big)^d2^{-pk}\right)^{10\log \chi}.$$
Applying Lemma \ref{densitylemma} below with $H=H_2$ to the coloring
$\Phi_0$ of graph $G'$ with partition $V_1 \cup \ldots \cup V_{k}$, we find an induced blue copy of $H_2$, completing
the proof. \hfill $\Box$

\begin{lemma}\label{densitylemma}
Let $H$ be a $d$-degenerate graph with vertex set $[k]$ such that
each vertex $i$ has at most $d$ neighbors less than $i$. Let
$G=(V,E)$ be a $(p,\lambda)$-pseudo-random graph on $n$ vertices
with $0<p \leq 3/4$, $\lambda \leq ((\frac{p}{10k})^d2^{-pk})^{10}n$
and let $V=V_1 \cup \ldots \cup V_k$ be a partition of its vertices
such that each $V_i$ has size $n/k$. Suppose that the edges of $G$
are $2$-colored, red and blue, such that for every edge $(j,\ell)$
of $H$, the density of red edges between the pair $(V_j,V_{\ell})$
is at most $\beta = \frac{p}{1000k^2}$. Then there is an induced
blue copy of $H$ in $G$ for which the image of every vertex $i \in
[k]$ lies in $V_i$.
\end{lemma}

\begin{proof}
For $i<j$, let $D(i,j)$ denote the number of neighbors of $j$ that are at most $i$. Let $\epsilon_1=\frac{1}{k}$,
$\epsilon_2=\frac{p}{10k}$, and $\delta=(1-p)^kp^d$. Since $p \leq 3/4$, notice that $\delta
\geq 2^{-3pk}p^d$  and

\begin{eqnarray}
\label{eq2}
\lambda \leq \left(\Big(\frac{p}{10k}\Big)^d2^{-pk}\right)^{10}n \leq \frac{p^8}{(10k)^{10}}\delta^2n.
\end{eqnarray}
We construct an induced blue
copy of $H$ one vertex at a time. At the end of step $i$, we will have vertices
$v_1,\ldots,v_i$ and subsets $V_{j,i} \subset V_j$ for $j >i $ such that the following four conditions hold
\begin{enumerate}
\item  for $j, \ell \leq i$, if $(j,{\ell})$ is an
edge of $H$, then $(v_j,v_{\ell})$ is a blue edge of $G$, otherwise
$v_j$ and $v_{\ell}$  are not adjacent in $G$,
\item for $ j \leq i < \ell$, if $(j,{\ell})$ is an
edge of $H$, then $v_j$ is adjacent to all vertices in $V_{\ell,i}$
by blue edges, otherwise there are no edges of $G$ from $v_j$ to
$V_{\ell,i}$,
\item for $i < j$, we have $|V_{j,i}| \geq (1-p-\epsilon_2)^{i-D(i,j)}(p-\epsilon_2)^{D(i,j)}|V_j|$,
\item and for $j, \ell>i$ if $(j,\ell)$ is an edge of $H$, then the density of red edges between $V_{j,i}$ and $V_{\ell,i}$ is
at most $(1+\epsilon_1)^i \beta$.
\end{enumerate}

Clearly, in the end of the first $k$ steps of this process we obtain a
required copy of $H$. For $i=0$ and $j \in [k]$, define
$V_{j,0}=V_j$. Notice that the above four properties are satisfied
for $i=0$ (the first two properties being vacuously satisfied). We
now assume that the above four properties are satisfied at the end
of step $i$, and show how to complete step $i+1$ by finding a vertex
$v_{i+1} \in V_{i+1,i}$ and subsets $V_{j,i+1} \subset V_{j,i}$ for $j>i+1$ such
that the conditions 1-4 still hold.

We need to introduce some notation. For a vertex $w \in V_j$ and a
subset $S \subset V_{\ell}$ with $j \not =\ell$, let
\begin{itemize}
\item $N(w,S)$ denote the set of vertices $s \in S$ such that $(s,w)$ is an edge of
$G$,
\item $B(w,S)$ denote the set of vertices $s \in S$ such that $(s,w)$ is a
blue edge of $G$,
\item $R(w,S)$ denote the set of vertices $s \in S$ such that $(s,w)$ is a
red edge of $G$,
\item $\tilde{N}(w,S)=N(w,S)$ if $(j, \ell)$ is an edge of $H$ and $\tilde{N}(w,S)=S
\setminus N(w,S)$ otherwise,
\item $\tilde{B}(w,S)=B(w,S)$ if $(j, \ell)$ is an edge of $H$ and $\tilde{B}(w,S):=S
\setminus N(w,S)$ otherwise, and
\item
$p_{j,\ell}=p$ if $(j,\ell)$ is an edge of $H$ and
$p_{j,\ell}=1-p$ if $(j,\ell)$ is not an edge of $H$.
\end{itemize}
Note that since graph $G$ is pseudo-random with edges density $p$, by the above definitions,
for every large subset $S \subset V_{\ell}$ and for most vertices $w \in V_j$ we expect
the size of $\tilde{N}(w,S)$ to be roughly  $p_{j,\ell}|S|$. We also have for
all $S \subset V_{\ell}$ and $w \in V_j$ that
$\tilde{B}(w,S)=\tilde{N}(w,S) \setminus R(w,S)$.

Call a vertex $w \in V_{i+1,i}$ {\it good} if for all $j >i+1$,
$\tilde{B}(w,V_{j,i})\geq (p_{i+1,j}-\epsilon_2)|V_{j,i}|$ and
for every edge $(j,\ell)$ of $H$ with $j, \ell>i+1$,
the density of red edges between $\tilde{B}(w,V_{j,i})$ and $\tilde{B}(w,V_{\ell,i})$ is at most
$(1+\epsilon_1)^{i+1}\beta$. If we find a good vertex $w \in V_{i+1,i}$, then we simply let
$v_{i+1}=w$ and $V_{j,i+1}=\tilde{B}(w,V_{j,i})$ for $j > i+1$,
completing step $i+1$. It therefore suffices to show that there is a
good vertex in $V_{i+1,i}$.

We first throw out some vertices of $V_{i+1,i}$ ensuring that the
remaining vertices satisfy the first of the two properties of good
vertices. For $j >i+1$ and an edge $(i+1,j)$ of $H$, let
$R_{j}$ consist of those $w \in V_{i+1,i}$ for which the
number of red edges $(w,w_j)$ with $w_j \in V_{j,i}$ is at least
$\frac{\epsilon_2}{2}|V_{j,i}|$. Since the density of red between
$V_{i+1,i}$ and $V_{j,i}$ is at most $(1+\epsilon_1)^i\beta$, then
$R_{j}$ contains at most
$$|R_{j}| \leq \frac{(1+\epsilon_1)^i\beta|V_{i+1,i}||V_{j,i}|}{\frac{\epsilon_2}{2}|V_{j,i}|}=2(1+\epsilon_1)^i\epsilon_2^{-1}\beta|V_{i+1,i}|$$
vertices. Let $V'$ be the set of vertices in $V_{i+1,i}$ that are not in
any of the $R_j$. Using that $\epsilon_1=1/k, \epsilon_2=\frac{p}{10k}$ and $
\beta=\frac{p}{1000k^2}$ we obtain
\begin{eqnarray*}
|V'| &\geq& |V_{i+1,i}|-\sum_{j>i+1}|R_j| \geq |V_{i+1,i}|-k\Big(2(1+\epsilon_1)^i\epsilon_2^{-1}\beta|V_{i+1,i}|\Big)\\
&\geq&
\Big(1-2k(1+\epsilon_1)^k\epsilon_2^{-1}\beta\Big)|V_{i+1,i}|
\geq \frac{1}{2}|V_{i+1,i}|.
\end{eqnarray*}

For $j >i+1$, let $S_j$ consist of those $w \in V'$
for which $\tilde{N}(w,V_{j,i}) < (p_{i+1,j}-\frac{\epsilon_2}{2})|V_{j,i}|$.
Then the density of edges of $G$ between
$S_j$ and $V_{j,i}$ deviates from $p$ by at least $\frac{\epsilon_2}{2}$.  Since
graph $G$ is $(p,\lambda)$-pseudo-random, we obtain that $\frac{\epsilon_2}{2} \leq
\frac{\lambda}{\sqrt{|V_{j,i}||S_j|}}$ and hence $|S_j| \leq \frac{4\lambda^2}{\epsilon_2^2|V_{j,i}|}$.
 Also using that $p \leq 3/4$ we have
$1-p-\epsilon_2=1-p-\frac{p}{10k} \geq (1-\frac{1}{3k})(1-p)$.
Therefore, our third condition, combined with $\delta=(1-p)^kp^d$ and $(1-x)^t \geq 1-xt$ for all $0 \leq x \leq 1$,
imply that for $j \geq i+1$
\begin{eqnarray}
\label{eq1}
|V_{j,i}| &\geq& (1-p-\epsilon_2)^{i-D(i,j)}(p-\epsilon_2)^{D(i,j)}|V_j| \geq
(1-p-\epsilon_2)^k(p-\epsilon_2)^d|V_j| \nonumber\\
 &\geq& \left(\Big(1-\frac{1}{3k}\Big)(1-p)\right)^k
\left(p-\frac{p}{10k}\right)^d|V_j| \nonumber \\
&\geq& \left(1-\frac{1}{3k}\right)^k\left(1-\frac{1}{10k}\right)^k(1-p)^kp^d|V_j| \nonumber\\
 &\geq&
\frac{1}{2}(1-p)^kp^d|V_i|= \frac{\delta n}{2k}.
\end{eqnarray}
Since $\lambda \leq \frac{p\delta}{100k^3}n$ (see (\ref{eq2}))
and $\epsilon_2=\frac{p}{10k}$, we therefore have $|S_j| \leq
\frac{4\lambda^2}{\epsilon_2^2|V_{j,i}|} \leq
\frac{1}{4k}|V_{i+1,i}|$. Let $V''$ be the set of vertices in $V'$
that are not in any of the sets $S_j$. The cardinality of $V''$ is at least
$$|V''|\geq |V'|-\sum_{j>i+1}|S_j| \geq |V'|-k\cdot\Big(\frac{1}{4k}|V_{i+1,i}|\Big) \geq
|V'|-\frac{1}{4}|V_{i+1,i}| \geq \frac{1}{4}|V_{i+1,i}|.$$
Moreover, by definition, for every $j>i+1$ and every vertex
$w \in V''$ there are  $|R(w,V_{j,i})| \leq \frac{\epsilon_2}{2}|V_{j,i}|$ red edges from
$w$ to $V_{j,i}$ if $(i+1,j)$ is an edge of $H$ and also $\tilde{N}(w,V_{j,i})$ has size at least $(p_{i+1,j}-\frac{\epsilon_2}{2})|V_{j,i}|$.
This implies that
$$|\tilde{B}(w,V_{j,i})|= |\tilde{N}(w,V_{j,i}) \setminus R(w,V_{j,i})| \geq |\tilde{N}(w,V_{j,i})|-\frac{\epsilon_2}{2}|V_{j,i}| \geq
(p_{i+1,j}-\epsilon_2)|V_{j,i}|$$
and therefore the vertices of $V''$ satisfy the first
of the two properties of good vertices.

We have reduced our goal to showing that there is an element of
$V''$ that has the second property of good vertices. For $i+1<j<\ell
\leq k$ and $(j,\ell)$ an edge of $H$, let $T_{j,\ell}$ denote the
set of $w \in V''$ such that the density of red edges between
$\tilde{B}(w,V_{j,i})$ and $\tilde{B}(w,V_{\ell,i})$ is more than
$(1+\epsilon_1)^{i+1}\beta$. Notice that any vertex of $V''$ not in
any of the sets $T_{j,\ell}$ is good. Therefore, if we show that
$T_{j,\ell} <\frac{|V''|}{k^2}$ for each $T_{j,\ell}$, then there is
a good vertex in $V''$ and the proof would be complete. To do so we
will assume without loss of generality that $p_{i+1,j}$ and
$p_{i+1,\ell}$ are both $p$ (the other 3 cases can be treated
similarly using the fact that $\bar G$ is
$(1-p,\lambda)$-pseudo-random). Since by (\ref{eq1}) we have that
$|V_{\ell,i}|, |V_{j,i}| \geq \frac{\delta n}{2k}$ and
$\frac{|V''|}{k^2} \geq \frac{1}{4k^2}|V_{i+1,i}| \geq \frac{\delta
n}{8k^3}$, the result follows from the following claim.

\begin{claim}
Let $X,Y$ and $Z$ be three disjoint subsets of our
$(p,\lambda)$-pseudo-random graph $G$ such that $|X| \geq
\frac{\delta n}{8k^3}$ and $|Y|,|Z| \geq \frac{\delta n}{2k}$. For
every $w \in X$ let $B_1(w), B_2(w)$ be the set of vertices in $Y$
and $Z$ respectively connected to $w$ by a blue edge and suppose
that $|B_1(w)|\geq (p-\frac{p}{10k})|Y|$ and $|B_2(w)| \geq
(p-\frac{p}{10k})|Z|$. Also suppose that the density of red edges
between $Y$ and $Z$ is at most $\eta$ for some $\eta \geq
\frac{p}{1000k^2}$. Then there is a vertex $w\in X$ such that the
density of red edges between $B_1(w)$ and $B_2(w)$ is at most
$\frac{k+1}{k}\eta$.
\end{claim}

\noindent {\bf Proof.} Let $m$ denote the number of triangles
$(x,y,z)$ with $x \in X, y \in Y, z \in Z$, such that the edge
$(y,z)$ is red. We need an upper bound on $m$. Let $U$ be the set of
vertices in $Y$ that have fewer than $p^3\delta^3(10k)^{-10}n$ red
edges to $Z$. So the number $m_1$ of triangles $(x,y,z)$ which have
$y\in U$ and edge $(y,z)$ red is clearly at most $m_1\leq
p^3\delta^3(10k)^{-10}n^3$. Let $W_1, W_2$ denote the subsets of
vertices in $Y$ whose number of neighbors in $X$ is at least
$(p+\frac{p}{20k})|X|$ or respectively at most
$(p-\frac{p}{20k})|X|$. Since the density of edges between $W_i$ and
$X$ deviates from $p$ by more than $\frac{p}{20k}$, using
$(p,\lambda)$-pseudo-randomness of $G$, we have $\frac{p}{20k} \leq
\frac{\lambda}{\sqrt{|X||W_i|}}$, or equivalently, $|X||W_i| \leq
400k^2p^{-2}\lambda^2.$ Therefore, using the upper bound $\lambda
\leq \frac{p^8}{(10k)^{10}}\delta^2n$ from (\ref{eq2}), the number
$m_2$ of triangles $(x,y,z)$ with $y \in W=W_1 \cup W_2$ and edge
$(y,z)$ red is at most
$$m_2 \leq |X||W|n \leq 800k^2p^{-2}\lambda^2n \leq (10k)^{-10}p^4\delta^4n^3.$$

For $y \in Y \setminus (U \cup W)$, we have the number of neighbors
of $y$ in $X$ satisfy $\big|\frac{|N(y, X)|}{|X|}-p\big|\leq
\frac{p}{20k}$ and the number of red edges from $y$ to $Z$ is at
least $p^3\delta^3(10k)^{-10}n$. Recall that $R(y,Z)$ denotes the
set of vertices in $Z$ connected to $y$ by red edges, hence we have
that $|R(y,Z)| \geq p^3\delta^3(10k)^{-10}n$ for every $y \in Y
\setminus (U \cup W)$. We also have that $|N(y, X)| \geq p|X|/2 \geq
\frac{p\delta n}{16k^3}$. Since $G$ is $(p,\lambda)$-pseudo-random,
we can bound the number of edges between $N(y, X)$ and $R(y,Z)$ by
$p|N(y, X)||R(y,Z)|+\lambda\sqrt{|N(y, X)||R(y,Z)|}$. Using the
above lower bounds on $|N(y, X)|$ and $|R(y,Z)|$, and the upper
bound (\ref{eq2}) for $\lambda$, one can easily check that
$$\frac{\lambda}{\sqrt{|N(y, X)||R(y,Z)|}}
\leq\frac{\lambda}{\sqrt{\big(p\delta n/(16k^3)\big)
\big(p^3\delta^3(10k)^{-10}n\big)}} \leq \frac{p}{20k}.$$ Hence the
number of edges between $N(y, X)$ and $R(y,Z)$ is at most
$(p+\frac{p}{20k})|N(y, X)||R(y,Z)|$. Recall that for all
$y \in Y \setminus (U \cup W)$ we have that
$|N(y, X)| \leq \big(p+\frac{p}{20k}\big)|X|$. Also, since the density
of red edges between $Y$ and $Z$ is at most $\eta$, we have that
$\sum_y|R(y,Z)| \leq \eta|Y||Z|$. Therefore, the
number $m_3$ of triangles $(x,y,z)$ with $y \in Y \setminus (U \cup
W), x\in X, z\in Z$ such that the edge $(y,z)$ is red is at most
$$m_3 \leq \Big(p+\frac{p}{20k}\Big)\sum_{y \in Y \setminus (U \cup W)}
 |N(y, X)||R(y,Z)|\leq
\Big(p+\frac{p}{20k}\Big)^2|X|\sum_y|R(y,Z)| \leq
\Big(p+\frac{p}{20k}\Big)^2\eta|X||Y||Z|.$$

Using the lower bounds on $|X|, |Y|, |Z|, \eta$ from the assertion of the claim we have that
$$p^2\eta|X||Y||Z| \geq \frac{p^3\delta^3}{(10k)^7}n^3 \geq (10k)^3 \max\big(m_1, m_2\big).$$
This implies that the total number of
triangles $(x,y,z)$ with $x \in X, y \in Y, z \in Z$, such that
the edge $(y,z)$ is red is at most
\begin{eqnarray*}
m &=& m_1+m_2+m_3 \leq
2\frac{p^2\eta|X||Y||Z|}{(10k)^3}+\Big(p+\frac{p}{20k}\Big)^2\eta|X||Y||Z|\\
&\leq& \big(1+1/(8k)\big)p^2\eta|X||Y||Z|.\end{eqnarray*} Therefore,
there is vertex $w \in X$ such that the number of these triangles
through $w$ is at most $(1+1/(8k))p^2\eta|Y||Z|$. Since $B_1(w)
\subset N(w,Y)$ and $B_2(w) \subset N(w,Z)$, then the number of red
edges between $B_1(w)$ and $B_2(w)$ is at most
$(1+1/(8k))p^2\eta|Y||Z|$. Since we have that $|B_1(w)|\geq
(p-\frac{p}{10k})|Y|$ and $|B_2(w)| \geq (p-\frac{p}{10k})|Z|$, the
density of red edges between $B_1(w)$ and $B_2(w)$ can be at most
$$\frac{(1+1/(8k))p^2\eta|Y||Z|}{|B_1(w)||B_2(w)|} \leq \frac{(1+1/(8k))p^2\eta} {(p-\frac{p}{10k})^2} \leq
\frac{k+1}{k}\eta,$$ completing the proof.
\end{proof}

\section{Trees with superlinear induced Ramsey numbers}
\label{superlinear}
In this section we prove Theorem \ref{tree}, that there are trees whose
induced Ramsey number is superlinear in the number of vertices. The proof
uses Szemer\'edi's regularity lemma, which we mentioned in the introduction.

A red-blue edge-coloring of the edges of a graph partitions the
graph into two monochromatic subgraphs, the {\it red graph}, which
contains all vertices and all red edges, and the {\it blue graph},
which contains all vertices and all blue edges. The weak induced
Ramsey number $r_{\textrm{weak ind}}(H_1,H_2)$, introduced by Gorgol
and \L uczak \cite{GoLu}, is the least positive integer $n$ such
that there is a graph $G$ on $n$ vertices such that for every
red-blue coloring of the edges of $G$, either the
red graph contains $H_1$ as an induced subgraph or the blue graph
contains $H_2$ as an induced subgraph. Note that this definition is a relaxation of the
induced Ramsey numbers since we allow blue edges  between the vertices of
red copy of $H_1$ or red edges between the vertices of blue copy of $H_2$.
Therefore a weak induced Ramsey number lies between the
usual Ramsey number and the induced Ramsey number. Using this new notion we
can strengthen Theorem \ref{tree} as follows. Recall that $K_{1,k}$ denotes a star with $k$ edges.

\begin{theorem}\label{weak}
For each $\alpha \in (0,1)$, there is a constant $k(\alpha)$ such that if $H$ is a graph on
$k \geq k(\alpha)$ vertices with maximum independent set of size less than
$(1-\alpha)k$, then $r_{\textrm{weak ind}}(H,K_{1,k}) \geq
\frac{k}{\alpha}$.
\end{theorem}
Let $T$ be a tree which is a union of path of length $k/2$ with the
star of size $k/2$ such that the end point of the path is the center
of the star. Since $T$ contains the path $P_{k/2}$ and the star
$K_{1,k/2}$ as induced subgraphs, then $r_{\textrm{ind}}(T) \geq
r_{\textrm{weak ind}}(P_{k/2},K_{1,k/2})$. By  using the above theorem with
$k/2$ instead of $k$, $H=P_{k/2}$, and sufficiently small
$\alpha$, we obtain that $r_{\textrm{ind}}(T)/k \rightarrow \infty$.
Moreover the same holds for every sufficiently large tree which
contains a star and a matching of linear size as subgraphs. We
deduce Theorem \ref{weak} from the following lemma.

\begin{lemma}\label{twocoloring} For each $\delta>0$ there is a constant $c_{\delta}>0$ such that if $G=(V,E)$ is a graph
on $n$ vertices, then there is a $2$-coloring of the edges of $G$
with colors red and blue such that the red graph has maximum degree
less than $\delta n$ and for every subset $W \subset V$, either there are
at least $c_{\delta}n^2$ blue edges in the subgraph induced by $W$
or there is an independent set in $W$ in the blue graph of
cardinality at least $|W|-\delta n$.
\end{lemma}
\begin{proof} Let $\epsilon=\frac{\delta^2}{100}$. By Szemer\'edi's
regularity lemma, there is a positive integer $M(\epsilon)$ together with
an equitable partition $V=\bigcup_{i=1}^k V_i$ of vertices of the graph $G=(V,E)$ into $k$ parts with
$\frac{1}{\epsilon}<k<M(\epsilon)$ such that all but at most $\epsilon
k^2$ of the pairs $(V_i,V_j)$ are $\epsilon$-regular. Recall that a partition is equitable if
$\big||V_i|-|V_j|\big| \leq 1$ and
a pair $(V_i,V_j)$ is called $\epsilon$-regular if for every $X \subset V_i$ and $Y
\subset V_j$ with $|X| > \epsilon |V_i|$ and $|Y| > \epsilon |V_j|$, we
have $|d(X,Y)-d(V_i,V_j)|<\epsilon$.
Let $c_{\delta}=\epsilon M(\epsilon)^{-2}$. Notice that to prove Lemma
\ref{twocoloring}, it suffices to prove it under the assumption that
$n $ is sufficiently large. So we may assume that $n \geq
\epsilon^{-1}M(\epsilon)$.

If a pair $(V_i,V_j)$ is $\epsilon$-regular with density
$d(V_i,V_j)$ at least $2\epsilon$, then color the edges between
$V_i$ and $V_j$ blue. Let $G'$ be the subgraph of $G$ formed by
deleting the edges of $G$ that are already colored blue. Let $V'$ be
the vertices of $G'$ of degree at least $\delta n$. Color blue any
edge of $G'$ with a vertex in $V'$. The remaining edges are colored
red. First notice that every vertex has red degree less than $\delta
n$.

We next show that $|V'|$ is small by showing that $G'$ has few
edges. There are at most
$$\sum_{i=1}^k {|V_i| \choose 2} \leq \frac{n^2}{k} \leq
\epsilon n^2$$ edges $(v,w)$ of $G$ with $v$ and $w$ both in the
same set $V_i$. Since at most $\epsilon k^2$ of the pairs
$(V_i,V_j)$ are not $\epsilon$-regular, then there are at most
$\epsilon n^2$ edges in such pairs. The $\epsilon$-regular pairs
$(V_i,V_j)$ with density less than $2\epsilon$ contain at most a
fraction $2\epsilon$ of all possible edges on $n$ vertices. So there
are less than $\epsilon n^2$ edges of this type. Therefore the
number of edges of $G'$ is at most $3\epsilon n^2$, and therefore
there are at most $|V'| \leq 2\frac{e(G')}{\delta n} \leq 6\epsilon
\delta^{-1}n<\frac{\delta n}{10}$ vertices of degree at least
$\delta n$ in it.

Let $W \subset V$. Let $W'=W \setminus V'$, so $W'$ has
cardinality at least $|W|-\frac{\delta n}{10}$. Let $W_i =V_i \cap
W'$. Let $W''=\bigcup_{|W_i| \geq \epsilon |V_i|} W_i$. Notice that
for any $i \in [k]$ there are at most $\epsilon \frac{n}{k}$
vertices in $(W' \setminus W'') \cap V_i$, so there are at most
$k(\epsilon \frac{n}{k}) =\epsilon n = \frac{\delta^2 n}{100}$
vertices in $W'\setminus W''$. Therefore, $W''$ has at least
$|W|-\delta n$ vertices. If there are $i \not =j$ such that
 $|W_i|,|W_j| \geq \epsilon\frac{n}{k}$ and the
pair $(V_i,V_j)$ is $\epsilon$-regular with density at least
$2\epsilon$, then there are at least
$$\epsilon|W_i||W_j| \geq \frac{\epsilon}{k^2}n^2 \geq \epsilon M(\epsilon)^{-2}n^2=c_{\delta}n^2$$ blue edges between $W_i$ and $W_j$.
In this case the blue subgraph induced by $W$
has at least $c_{\delta}n^2$ edges. Otherwise, all the edges in
$W''$ are red, and $W''$ is an independent set in the blue graph of
cardinality at least $|W|-\delta n$.
\end{proof}

\vspace{0.15cm} \noindent {\bf Proof of Theorem \ref{weak}.}\, Let
$H$ be a graph on $k$ vertices with maximum independent set of size
less than $(1-\alpha)k$. Take $\delta=\alpha^2$ and  $c_{\delta}$ to
be as in Lemma \ref{twocoloring}. Let $G=(V,E)$ be any graph on $n$
vertices, where $n \leq \frac{k}{\alpha}$. If $H$ has at least
$c_{\delta}k^2$ edges, consider a random red-blue coloring of the
edges of $G$ such that the probability of an edge being red is
$\frac{\alpha}{2}$. The expected degree of a vertex in the red graph
is at most $\alpha n/2$. Therefore by the standard Chernoff bound
for the Binomial distribution it is easy to see that with
probability $1-o(1)$ the degree of every vertex in the red graph is
less than $\alpha n\leq k$, i.e., it contains no $K_{1,k}$. On the
other hand, for $k$ sufficiently large, the probability that the
blue graph contains a copy of $H$ is at most
$$ n^k (1-\alpha/2)^{e(H)} \leq n^k e^{-\alpha c_{\delta}k^2/2} \leq
e^{-\alpha c_{\delta}k^2/2+k\log (k/\alpha)}=o(1).$$
Thus with high probability this coloring has no blue copy of $H$ as well.
This implies that we can assume that the number of edges in $H$ is less than $c_{\delta}k^2$.

By Lemma \ref{twocoloring}, there is a red-blue edge-coloring of the
edges of $G$ such that the red graph has maximum degree at most
$\delta n$ and every subset $W \subset V$ contains either an independent
set in the blue graph of size at least $|W|-\delta n$ or contains at
least $c_{\delta}n^2$ blue edges. Since $\delta n=\alpha^2 n <k$, then the red
graph does not contain $K_{1,k}$ as a subgraph.
Suppose for contradiction that there is an induced copy of $H$ in
the blue graph, and let $W$ be the vertex set of this copy.
The blue graph induced by $W$ has
$e(H)<c_{\delta}k^2\leq c_{\delta}n^2$ edges.
Therefore it contains an independent set of size at least $|W|-\delta n \geq
|W|-\alpha k=(1-\alpha)k$, contradicting the fact that $H$ has no
independent set of size $(1-\alpha)k$. Therefore, there are no
induced copies of $H$ in the blue graph.\qed

\section{Concluding Remarks}
\begin{itemize}
\item
All of the results in this paper concerning induced subgraphs can be extended to
many colors. One such multicolor result
was already proved in Section \ref{moreoninduced} (see Lemma \ref{lemmaerdoshajnal2}),
and we use here the notation from that section. For example, one can obtain the following
generalization of Theorem \ref{main}. For $k \geq 2$, let $\Psi:E(K_{k}) \rightarrow [r]$ be
an edge-coloring of the complete graph $K_{k}$ and $\Phi:E(K_n) \rightarrow [s]$ be
a $\Psi$-free edge-coloring of the complete graph $K_n$. Then there is a constant $c$ so that for
every $\epsilon \in (0,1/2)$,  there is a subset $W \subset K_n$ of size at least $2^{-crk (\log
\frac{1}{\epsilon})^2}n$ and a color $i \in [r]$ such that the edge
density of color $i$ in $W$ is at most $\epsilon$.
Since the proofs of this statement and other generalizations can be obtained using our key lemma in essentially the same way as
the proofs of the results that we already presented (which correspond to the two color case), we
do not include them here.

\item
It would be very interesting to get a better estimate in Theorem \ref{main}. This will immediately
give an improvement of the best known result for Erd\H{o}s-Hajnal conjecture on the size of the maximum
homogeneous set in $H$-free graphs. We believe that our bound can be strengthened as follows.

\begin{conjecture}\label{mainconjecture}
For each graph $H$, there is a constant $c(H)$ such that if
$\epsilon \in (0,1/2)$ and $G$ is a $H$-free graph on $n$ vertices,
then there is an induced subgraph of $G$ on at least
$\epsilon^{c(H)}n$ vertices that has edge density either at most
$\epsilon$ or at least $1-\epsilon$.
\end{conjecture}

\noindent
This conjecture if true would imply the Erd\H{o}s-Hajnal
conjecture. Indeed, take $\epsilon=n^{-\frac{1}{c(H)+1}}$. Then
every $H$-free graph $G$ on $n$ vertices contains an induced
subgraph on at least $\epsilon^{c(H)}n=n^{\frac{1}{c(H)+1}}$
vertices that has edge density at most $\epsilon$ or at least
$1-\epsilon$. Note that this induced subgraph or its
complement has average degree at most $1$, which implies that it
contains a clique or independent set of size at least
$\frac{1}{2}n^{\frac{1}{c(H)+1}}$.

\item
One of the main remaining open problems on induced Ramsey numbers is
a beautiful conjecture of Erd\H{o}s which states that there exists a
positive constant $c$ such that $r_{\textrm{ind}}(H) \leq 2^{ck}$
for every graph $H$ on $k$ vertices. This, if true, will show that
induced Ramsey numbers in the worst case have the same order of
magnitude as ordinary Ramsey numbers. Our results here suggest that
one can attack this problem by studying 2-edge-colorings of a random
graph with edge probability $1/2$. It looks very plausible that for
sufficiently large constant $c$, with high probability random graph
$G(n,1/2)$ with $n\geq 2^{ck}$ has the property that any of its
2-edge-colorings contains every graph on $k$ vertices as an induced
monochromatic subgraph. Moreover, maybe this is even true for every
sufficiently pseudo-random graph with edge density $1/2$.

\item The results on induced Ramsey numbers of sparse graphs naturally
lead to the following questions. What is the asymptotic
behavior of the maximum of induced Ramsey numbers over all trees on
$k$ vertices? We have proved $r_{\textrm{ind}}(T)$ is superlinear in
$k$ for some trees $T$. On the other hand, Beck
\cite{Be} proved that $r_{\textrm{ind}}(T)=O\left( k^{2}\log^2
k\right)$ for all trees $T$ on $k$ vertices.

For induced Ramsey numbers of bounded degree graphs, we proved a
polynomial upper bound with exponent which is nearly linear in the
maximum degree. Can this be improved further, e.g., is it true that
the induced Ramsey number of every $n$-vertex graph with maximum
degree $d$ is at most a polynomial in $n$ with exponent independent
of $d$? It is known that the usual Ramsey numbers of bounded degree
graphs are linear in the number of vertices.
\end{itemize}

\vspace{0.2cm} \noindent {\bf Acknowledgment.}\, We'd like to thank
Janos Pach and Csaba T\'oth for helpful comments on an early stage
of this project and Steve Butler and Philipp Zumstein for carefully
reading this manuscript.

\end{document}